\documentclass[final]{article}

\title{Biorthogonal Rational Krylov Subspace Methods}

\author{ Niel Van Buggenhout\footnotemark[2]
	\and Marc Van Barel\footnotemark[2]
	\and Raf Vandebril\footnotemark[2]}

\usepackage{amsmath}
\usepackage{amssymb}
\usepackage{amsthm}
\usepackage{float}
\usepackage{graphicx}
\usepackage{subcaption}
\usepackage{cite}
\usepackage{tikz}
\usepackage{tikzscale}
\newlength\figureheight
\newlength\figurewidth
\usepackage{pgfplots}
\usetikzlibrary{matrix}
\usepackage{rotfig}
\usepackage{mathdots}

\newtheorem{theorem}{Theorem}[section]
\newtheorem{remark}{Remark}[section]
\newtheorem{lemma}{Lemma}[section]
\newtheorem{example}{Example}[section]

\usepackage[section]{algorithm}
\usepackage{algpseudocode}

\usepackage{makecell}
\usepackage{multirow}

\newtheorem{property}{Property}[section]
\DeclareMathOperator{\spn}{span}

\DeclareMathOperator{\struct}{struct}
\DeclareMathOperator{\x}{x}
\newcommand{\temp}{\mathrm{temp}}
\newcommand{\tempv}{\mathrm{tempv}}
\newcommand{\tempw}{\mathrm{tempw}}

\newcommand{\K}{\mathcal{K}}
\newcommand{\Ls}{\mathcal{L}}
\newcommand{\Cmm}{\mathbb{C}^{m\times m}}
\newcommand{\Cm}{\mathbb{C}^{m}}
\newcommand{\Cb}{\bar{\mathbb{C}}}
\newcommand{\bv}{{\bar{v}}}
\newcommand{\bw}{{\bar{w}}}
\newcommand{\tv}{{\tilde{v}}}
\newcommand{\tw}{{\tilde{w}}}
\newcommand{\hv}{{\hat{v}}}
\newcommand{\hw}{{\hat{w}}}
\newcommand{\Hi}{H^{\textrm{inv}}}

\newcommand{\Ki}{K^{\textrm{inv}}}

\newcommand{\fakesection}[1]{%
	\par\refstepcounter{section}% Increase section counter
	\sectionmark{#1}% Add section mark (header)
	\addcontentsline{toc}{section}{\protect\numberline{\thesection}#1}% Add section to ToC
}

\providecommand{\keywords}[1]{\textit{Keywords: } #1}
\date{}
\begin{document}
	\maketitle

	\renewcommand{\thefootnote}{\fnsymbol{footnote}}
	
	\footnotetext[2]{Department of Computer Science, KU Leuven, University of Leuven, 3001 Leuven, Belgium. (raf.vandebril@kuleuven.be, niel.vanbuggenhout@kuleuven.be, marc.vanbarel@kuleuven.be)}
	\footnotetext{The research of the authors was supported by
		the Research Council KU Leuven,
		C1-project (Numerical Linear Algebra and Polynomial Computations),
		by
		the Fund for Scientific Research--Flanders (Belgium), EOS Project no 30468160, and by
		the Research Council KU Leuven:
		C14/16/056 (Inverse-free Rational Krylov Methods: Theory and Applications)}
	
	\begin{abstract}
		A general framework for oblique projections of nonhermitian matrices onto rational Krylov subspaces is developed. 
		To obtain this framework we revisit the
		classical rational Krylov subspace algorithm and prove that the projected matrix can be
		written efficiently as a structured pencil, where the structure can take several
		forms, such as Hessenberg or inverse Hessenberg.
		One specific instance of the structures appearing in this framework for oblique projections is a tridiagonal pencil. This is a direct generalization of the classical biorthogonal Krylov subspace method where the projection becomes a single nonhermitian tridiagonal matrix and of the Hessenberg pencil representation for rational Krylov subspaces. 
		Based on the compact storage of this tridiagonal pencil in the biorthogonal setting,
			we can develop short recurrences.
			Numerical experiments confirm the validity of the approach.
	\end{abstract}
%\begin{keywords}
%	Rational Krylov, Biorthogonal, Short Recurrence, Oblique Projection, Matrix Pencil
%\end{keywords}
\keywords{Rational Krylov, Biorthogonal, Short Recurrence, Oblique Projection, Matrix Pencil}

%\begin{AMS}
%	15A22, 47A75, 65F99, 65Q30
%\end{AMS}

	\section{Introduction}
	
	Krylov subspace methods, introduced by A. N.~Krylov \cite{Kr31}, are an indispensable tool in science and engineering for
	transforming large datasets to manageable sizes. There is an
	enormous amount of variants of Krylov subspace methods. A good overview can be found in the books
	of Saad \cite{Sa03}, van der Vorst \cite{vdV03}, and Gutknecht \cite{Gu97}. In
	this article we focus on a particular type of Krylov subspace methods, namely the rational
	Krylov subspace methods in a non-orthogonal, but oblique projection process.
	This allows to save the projected matrix as two tridiagonal matrices, which is a more data-sparse representation compared to the Hessenberg pair arising from orthogonal projection.

	Rational Krylov subspaces were introduced by Ruhe \cite{Ru84}
	illustrating that faster convergence could be obtained when, e.g., approximating non-dominating eigenvalues \cite{Ru94} and constructing a reduced-order model for dynamical systems \cite{GaGrVD94,GaGrVD96,Gr97}.
	
	Arnoldi \cite{Ar51} linked Hessenberg matrices to the orthogonal basis stemming from a
	Krylov subspace and developed an iteration to build up these bases.  Iterative construction of matrices involved in biorthogonal Krylov subspace methods
	are due to Lanczos \cite{La50}, where an oblique
	projection results in a tridiagonal matrix.
	Even though the oblique projection process is less stable than the classical orthogonal
	projection, there is a significant gain in memory storage and computing time.  A nice
	introduction into biorthogonal Krylov subspace methods is provided by Saad \cite{Sa92}. The most popular biorthogonal method for solving systems of equations
	is the BiCGStab method of van der Vorst \cite{Vo92}.
	
	Biorthogonal Krylov subspace methods for rational Krylov subspaces have been described only partially in literature \cite{GaGrVD94,GaGrVD96,Gr97}. This article will generalize previous results and provide a general framework. We will prove that the oblique projection linked to biorthogonal
	rational Krylov subspaces results in a matrix pencil, of which both matrices can be chosen
	to be tridiagonal, possibly nonhermitian. The highly structured pencil allows us to develop a
	short recursion to compute the biorthogonal bases and the projected pencil.
	To derive these results we first need to reconsider the structure of the orthogonally projected matrix
	linked to a classical rational Krylov subspace. We prove that instead of the
	single rational Hessenberg matrix we can also work with a pair of matrices of particular
	structure, such as Hessenberg or inverse Hessenberg.
	
	Gutknecht studied short recursions, $(k,l)$-step methods for fixed point equations \cite{Gu15,Gu89}, by means of a Hessenberg-triangular pencil. Some classical Krylov subspace methods can be described by $(k,l)$-step methods, e.g., BiCG is a $(2,1)$-step method. The biorthogonal rational Lanczos method introduced here does not immediately fit this framework.
	
	Some notable results are provided below, which are in some sense special cases of the general framework provided here. We discuss which spaces are used and what structure the projection onto these subspaces exhibits. 
	
	Using orthogonality of Laurent polynomials, Jagels and Reichel \cite{JaRe11} constructed a recurrence for extended Krylov subspaces with regularity in the poles (a repetition of $i\geq 1$ times $A$ and one time $A^{-1}$) and a symmetric matrix $A$. They represented their projected matrix as a single matrix. Schweitzer \cite{Sc17} constructed in a similar way a nonsymmetric Lanczos iteration for extended Krylov subspaces, only valid when a negative and positive power is alternated in both spaces.
	Gallivan, Grimme and Van Dooren derived a nonsymmetric rational Lanczos iteration \cite{GaGrVD96}. They use the same poles in both subspaces and represent the projection as a pencil, which is a tridiagonal pencil, except for some off-diagonal fill-in when a change of pole occurs.
	Watkins \cite{Wa93} provided the first elegant representation of the AGR/CMV-factorization \cite{AmGrRe86,BuEl91,Wa93,CaMoVe03,Si07} as a matrix pencil, for a nice overview of the history we refer to the paper by Simon \cite{Si07}. This factorization is in fact also a biorthogonal relation, but for unitary matrices.
	
%	It is interesting to note that the generalization as proposed in this article nicely
%	covers classical results linked to the AGR- or CMV-factorization \cite{AmGrRe86,BuEl91,Wa93,CaMoVe03,Si07}.
%	Which state that the orthogonal projection of a unitary matrix onto an extended
%	Krylov subspace \cite{DrKn98} can be factored as the product of two tridiagonal matrices. It are these
%	tridiagonal matrices that are linked to the pencil resulting from the biorthogonal
%	projection process.

	Some elementary results are provided in Section \ref{sec:basics}, with a focus on sparsity and low-rank structure.
	Section \ref{sec:EKS} discusses rational Krylov subspace methods and the structure of the projection.
	Section \ref{sec:BiExt} deals with biorthogonal rational Krylov subspace methods and an overview presenting all possible structures.
	In Section \ref{sec:RatLan} a rational Lanczos iteration is derived based on the tridiagonal pencil structure, some numerical experiments are performed illustrating the validity of the approach.

	\section{Basics}\label{sec:basics}
	Since this text will rely on matrix computations and the main results involve sparsity and low-rank structure, this section is devoted to these types of structure (structure will refer from now on to both sparsity and low-rank structure).
	Useful elementary results for standard Krylov subspace methods are repeated in Section \ref{sec:SKS}. For more details see, e.g., \cite{LiSt13,Sa92,Pa71}.
	Using the QR-factorization we introduce inv-Hessenberg, extended Hessenberg and rational Hessenberg matrices in Section \ref{sec:structure}.
	
	\subsection{Standard Krylov subspaces} \label{sec:SKS}
	Standard Krylov subspace methods perform an orthogonal projection of some matrix $A\in \Cmm$ onto the Krylov subspace
	 $$\mathcal{K}_n(A,v) = \spn\{v,Av,A^2v,\dots,A^{n-1}v \},$$
	 with a starting vector $v\in \mathbb{C}^{m}$, $\Vert v \Vert_2=1$. Note that these subspaces are nested, i.e., $\mathcal{K}_{n-1} \subseteq \mathcal{K}_{n} $. Using the Arnoldi iteration \cite{Ar51} a \emph{nested} orthonormal basis $V_n$ for $\mathcal{K}_n$ can be iteratively constructed together with the projection onto the lower dimensional subspace $\mathcal{K}_{n-1}(A,v)$: $V_{n-1}^HAV_{n-1}=H_{n-1}\in \mathbb{C}^{(n-1)\times (n-1)}$. \\
	A basis $V_n \in \mathbb{C}^{m\times n}$ for a subspace $\mathcal{S}_n$ of dimension $n$ is called nested if $\mathcal{S}_1\subseteq \mathcal{S}_2 \subseteq \mathcal{S}_3 \subseteq \dots$, where $\mathcal{S}_i$ is spanned by the first $i$ columns of $V_n$.
	%its first $i<j\leq n$ columns form a basis for lower dimensional subspaces $\mathcal{S}_i$ such that $\mathcal{S}_i \subseteq \mathcal{S}_j,$ for $i<j\leq n$.\\ %, where $(.)^H$ denotes the Hermitian conjugate. \\
	The projected matrix $H_n$ has upper-Hessenberg structure, i.e., $h_{i,j} = 0$ for $i>j+1$, where $h_{i,j}$ denotes the element on the $i$th row and $j$th column of $H_n$. An alternative notation that will be used is $(H_n)_{i,j}$.
	In general $H_n$ exhibits no particular structure above its diagonal.
	\begin{remark}
			Throughout this text we assume that no breakdowns occur, i.e., no subdiagonal element $h_{i+1,i}$ of the projection $H_{i+1} = V_{i+1}^HAV_{i+1}$ is zero. Since, a zero would imply that the subspace $\mathcal{K}_i$ is invariant under multiplication with $A$ or in other words $A\mathcal{K}_i=\mathcal{K}_i$. Here every occasion where it is impossible to expand the current subspace $\mathcal{S}_i$, i.e., $\mathcal{S}_{i+1} = \mathcal{S}_i$ will be called a breakdown. Typically a breakdown is a lucky event, i.e., lucky termination and we will therefore not focus on it. Serious breakdowns can also occur, see Gutknecht \cite{Gu97} and references therein for details.
	\end{remark}

	For a full reduction, i.e., $n=m$ the subscripts are dropped $V^HAV = H$, $H \in \mathbb{C}^{m\times m}$.\\
	The structure of $H_n$ can be represented as shown in Figure \ref{fig:Hess_example}, where $\textrm{struct}(M)$ of some matrix $M$ shows generic nonzero elements as $\times$ and omits the zeros.
	In case of a Hermitian matrix $A^H = A$, the orthogonal projection onto $\mathcal{K}_n(A,v)$ results in a Hermitian Hessenberg matrix $V_n^HAV_n = T_n$. Or in other words it has Hessenberg structure both above and below its diagonal and is therefore tridiagonal, which is shown in Figure \ref{fig:Hess_example}.
	Since we assumed no breakdowns, the Hessenberg and tridiagonal matrix are both \emph{proper}, i.e., no zeros appear on the subdiagonal.
	\begin{figure}[!ht]
		\centering
		\begin{subfigure}[b]{0.4\textwidth}
			\centering
			{	\begin{tikzpicture} 
 \matrix (M) [matrix of nodes,left delimiter={[},right delimiter={]}] 
{$\times$&$\times$&$\times$&$\times$&$\times$&$\times$&$\times$\\
$\times$&$\times$&$\times$&$\times$&$\times$&$\times$&$\times$\\
&$\times$&$\times$&$\times$&$\times$&$\times$&$\times$\\
&&$\times$&$\times$&$\times$&$\times$&$\times$\\
&&&$\times$&$\times$&$\times$&$\times$\\
&&&&$\times$&$\times$&$\times$\\
&&&&&$\times$&$\times$\\
};
% \draw [line width=0.25mm] (M-2-1.center) -- (M-7-6.center); 
% \node[text width=0.1cm] at (2.5,0) {{$,$}};
\end{tikzpicture}}
			\caption{$\struct(H_n)$}
		\end{subfigure}
		\begin{subfigure}[b]{0.4\textwidth}
			\centering
			{	\begin{tikzpicture} 
\matrix (M) [matrix of nodes,left delimiter={[},right delimiter={]}] 
{$\times$&$\times$&&&&&\\
	$\times$&$\times$&$\times$&&&&\\
	&$\times$&$\times$&$\times$&&&\\
	&&$\times$&$\times$&$\times$&&\\
	&&&$\times$&$\times$&$\times$&\\
	&&&&$\times$&$\times$&$\times$\\
	&&&&&$\times$&$\times$\\
};
%\draw [line width=0.25mm] (M-2-1.center) -- (M-7-6.center); 
%\draw [line width=0.25mm] (M-1-2.center) -- (M-6-7.center); 
\end{tikzpicture}}
			\caption{$\struct(T_n)$}
		\end{subfigure}
		\caption{Generic nonzero elements of a Hessenberg matrix $H_n$ and tridiagonal matrix $T_n$ are shown as $\times$.}
		\label{fig:Hess_example}
	\end{figure}
	\subsection{Sparsity and low-rank structure}\label{sec:structure}
	The sparsity that a Hessenberg matrix exhibits below its diagonal is also contained in its QR-factorization. The QR-factorization decomposes a matrix into the product of a unitary matrix $Q$ and upper-triangular matrix $R$.\\
	To discuss the QR-factorization of a Hessenberg matrix we require core transformations. \emph{Core transformations} in this text will refer to unitary matrices $C_i$ that equal the unit matrix with a $2\times 2$ unitary block embedded on the diagonal starting in row and column $i$:
	%\begin{figure}[!ht]
	%	\centering
	%	{\input{figs/core_transf.tikz}}
	%\end{figure}
	\begin{equation*}
	%	C_i = \begin{bmatrix}
	%		1 \\
	%		& \ddots \\
	%		 & & 1\\
	%		 & & & \alpha & \beta \\
	%		 & & & \gamma & \delta \\
	%		 & & & & & 1\\
	%		 & & & & & & \ddots\\
	%		 & & & & & & & 1
	%
	%	\end{bmatrix}
	C_i = \begin{bmatrix}
	I_{i-1} \\
	& \times & \times \\
	& \times & \times \\
	& & & I_{n-i-1}\\
	\end{bmatrix},
	\end{equation*}
	where $C_i$ is of size $n\times n$ and $I_{k}$ denotes the unit matrix of size $k\times k$.
	To compactly visualize a core transformation $C_i$, the notation
	$\footnotesize
	\begin{array}{cccc}
	\Rc \\
	\rc\end{array}$ will be used, with the top arrow pointing to row $i$ and the bottom arrow pointing to row $i+1$.
	Multiplication from the left with a core transformation $C_i$: $C_iM$, only affects the $i$th and $(i+1)$th rows of the matrix $M$.\\
	
	\begin{lemma} [QR-factorization of Hessenberg matrices]\label{lemma:QR_Hess}
		Consider a proper upper-Hessenberg matrix $H\in \mathbb{C}^{n\times n}$, $h_{i,j}=0$ for $i>j+1$, the QR-factorization of $H$ can be written as $H=C_1C_2\dotsm C_{n-1}R$, where the $C_i$ are nontrivial core transformations.
	\end{lemma}
	
	We will refer to $C_1C_2\dotsm C_{n-1}$ as a \emph{descending pattern} of core transformations. In case of an \emph{ascending pattern} $Q=C_{n-1}\dotsm C_2C_{1}$, $QR$ forms an \emph{inv-Hessenberg matrix}. Inv-Hessenberg matrices have a low-rank structure below their diagonal similar to the structure of inverse Hessenberg matrices. We distinguish them from inverse Hessenberg matrices, since they do not have to be invertible. More details can be found in, e.g., the book by Vandebril et al. \cite{VaVBMa07}, where they are called Hessenberg-like matrices.\\
	Now a logical next step is to look at the structure of $Z=QR$ if the shape (the ordering of core transformations) contains ascending and descending patterns, i.e., a permutation of $C_1 C_2 \dotsm C_{n-1}$. To be able to discuss this we note that $C_iC_j=C_jC_i$, for $\vert i-j\vert>1$.	
	Whenever a descending pattern $C_iC_{i+1}$ occurs, a Hessenberg block is formed and whenever an ascending pattern $C_{i+1}C_i$ occurs, an inv-Hessenberg block is formed.
	\begin{example}
		Take, for example, $Q=C_1C_4C_5C_6C_3C_2$, corresponding to the shape
		\begin{equation}
		\begin{array}{cccc}
		& \Rc \\
		& \rc & \Rc\\
		& \Rc&   \rc\\
		\Rc& \rc\\
		\rc & \Rc \\
		& \rc & \Rc\\
		& & \rc
		\end{array}.
		\end{equation}
		The structure of $Z=QR$ is visualized similarly as by Mertens and Vandebril \cite{MeVa15} in Figure \ref{fig:QR_example}. The dashed and dotted lines highlight the structure.\\
		\begin{figure}[!ht]
			\centering
			{\begin{tikzpicture} 
\node[text width=3cm] at (-2.4,0) {{$\textrm{struct}(Z)=$}};
 \matrix (M) [matrix of nodes,left delimiter={[},right delimiter={]}] 
{$\times$&$\times$&$\times$&$\times$&$\times$&$\times$&$\times$\\
$\times$&$\times$&$\times$&$\times$&$\times$&$\times$&$\times$\\
&$\times$&$\times$&$\times$&$\times$&$\times$&$\times$\\
&$\times$&$\times$&$\times$&$\times$&$\times$&$\times$\\
&$\times$&$\times$&$\times$&$\times$&$\times$&$\times$\\
&&&&$\times$&$\times$&$\times$\\
&&&&&$\times$&$\times$\\
};
% Hess block
\draw [thick,dotted] (M-2-1.center) -- (M-1-1.center); 
\draw [thick,dotted] (M-2-1.center) -- (M-3-2.center); 
\draw [thick,dotted] (M-3-2.center) -- (M-3-3.center); 
\draw [thick,dotted] (M-3-3.center) -- (M-1-3.center); 
\draw [thick,dotted] (M-1-3.center) -- (M-1-1.center); 

% inv Hess block
\draw [thick,dashed] (M-2-2.center) -- (M-5-2.center); 
%\draw [thick,dotted] (M-2-3.center) -- (M-4-5.center); 
\draw [thick,dashed] (M-5-5.center) -- (M-5-4.center); 
\draw [thick,dashed] (M-5-4.center) -- (M-5-2.center); 
\draw [thick] (M-2-2.center) -- (M-2-5.center); 
\draw [thick] (M-2-5.center) -- (M-5-5.center); 
\draw [thick,dashed] (M-2-2.center) -- (M-5-5.center); 
 
% Hess block
\draw [thick,dotted] (M-5-4.center) -- (M-4-4.center); 
\draw [thick,dotted] (M-4-4.center) -- (M-4-7.center); 
\draw [thick,dotted] (M-4-7.center) -- (M-7-7.center); 
\draw [thick,dotted] (M-7-7.center) -- (M-7-6.center); 
\draw [thick,dotted] (M-7-6.center) -- (M-5-4.center);

\end{tikzpicture}}
			\caption{Extended Hessenberg matrix $Z$ obtained by a shape of core transformations $Q=C_1C_4C_5C_6C_3C_2$, such that $Z=QR$.}
			\label{fig:QR_example}
		\end{figure}
		
		From the structure in Figure \ref{fig:QR_example} and the corresponding shape of $Q$ it is clear that $C_1C_2$ forms a Hessenberg block $Z_{1:3,1:3}$ (indicated by a dotted line), $C_4C_3C_2$ forms an inv-Hessenberg block $Z_{2:5,2:5}$ (low rank part indicated by a dashed line) and $C_4C_5C_6$ forms again a Hessenberg block $Z_{4:7,4:7}$.\\
	\end{example}
	
	A matrix containing both ascending and descending patterns of core transformations will be referred to as an extended Hessenberg matrix and links to the projection onto an extended Krylov subspace \cite{Va11}, which is a special case of a rational Krylov subspace.

	\section{Rational Krylov subspaces}\label{sec:EKS}
	Rational Krylov subspaces \cite{Ru84} will be denoted by $\K(A,v;\Xi)$, where $A$ and $v$ are defined as before and poles $\Xi = \{\xi_1, \xi_2,\dots\}$, with $\xi_k \in \Cb = \{\mathbb{C}\cup {\infty}\}$. If the $k$th pole is finite, a shift-invert operator $(\nu_k A-\mu_k I)^{-1}$ expands the subspace $\K_k(A,v,\{\xi_i\}_{i=1}^{k-1})$. The ratio $\mu_k/\nu_k = \xi_k$, which is the $k$th pole. We call this a pole since it is the pole of the shift-invert operator $(A-\xi_kI)^{-1}$. If the $k$th pole is infinite, multiplication with $A$ expands the subspace.
	First the single-matrix representation of the orthogonal projection onto a rational Krylov subspace is considered in Section \ref{sec:OrthExt_single} and afterwards the pencil representation of this projection is discussed in Section \ref{sec:OrthExt_pair} for a Hessenberg pencil, and in Section \ref{sec:inv_Hess} for an inv-Hessenberg pencil. %The pencil representation is paramount to the derivation of the results in Section \ref{sec:BiExt}.
	
	\subsection{Single-matrix representation}\label{sec:OrthExt_single}
	Consider an orthonormal nested basis $V_n \in \mathbb{C}^{m \times n}$ for $\K_n(A,v;\Xi)$, with $A \in \mathbb{C}^{m\times m}$, $v\in \mathbb{C}^m$ and given poles $\Xi$. Orthogonally projecting the matrix $A$ onto $\K_n$ and expressing the result using a single matrix $Z_n$ provides the equation
	\begin{equation}
	V_n^HAV_n=Z_n.
	\end{equation}
	The structure of the \emph{rational Hessenberg matrix} $Z_n$ can be deduced from the choice of poles. It allows for a factorization as $Z_n = QR+D$, where $QR$ forms an extended Hessenberg matrix and $D$ is a diagonal matrix containing the poles of the corresponding rational Krylov subspace \cite{CaMeVa19}. Expansion using a shift-invert operator (finite pole) leads to an inv-Hessenberg block and expansion using multiplication with $A$ (infinite pole) leads to a Hessenberg block \cite{AuMaRoVaWa18}. Example \ref{example:blocks} illustrates this.
	\begin{example}\label{example:blocks}
		Consider the extended Krylov subspace corresponding to the example from before $Z = C_1C_4C_5C_6C_3C_2R$, shown in Figure \ref{fig:QR_example}, $$\K_{7}=\spn\{v,Av,A^{-1}v,A^{-2}v,A^2v,A^3v,A^4v \}.$$
		The corresponding poles are $\Xi = \{\infty, 0, 0, \infty, \infty , \infty\}$.\\
		If the poles are chosen as $\Xi = \{\infty, \xi_2 = \frac{\mu_2}{\nu_2}, \xi_3 = \frac{\mu_3}{\nu_3}, \infty, \infty , \infty\}$ , the space constructed is $$\spn\{v,Av,(\nu_2 A-\mu_2 I)^{-1}v,(\nu_2 A-\mu_2 I)^{-1}(\nu_3 A-\mu_3 I)^{-1}v,A^2v,A^3v,A^4v \}$$ and the decomposition becomes $Z=QR+D$ as shown on figure \ref{fig:QRplusD}, where $\xi_2$ and $\xi_3$ are the poles and the remaining elements of $D$ can be chosen freely.
			\begin{figure}[!ht]
				\centering
				{\begin{tikzpicture} 
\node[text width=3cm] at (-2.4,0) {{$\textrm{struct}(Z)=$}};
 \matrix (M) [matrix of nodes,left delimiter={[},right delimiter={]}] 
{$\times$&$\times$&$\times$&$\times$&$\times$&$\times$&$\times$\\
$\times$&$\times$&$\times$&$\times$&$\times$&$\times$&$\times$\\
&$\times$&$\times$&$\times$&$\times$&$\times$&$\times$\\
&$\times$&$\times$&$\times$&$\times$&$\times$&$\times$\\
&$\times$&$\times$&$\times$&$\times$&$\times$&$\times$\\
&&&&$\times$&$\times$&$\times$\\
&&&&&$\times$&$\times$\\
};
% Hess block
\draw [thick,dashed] (M-2-1.center) -- (M-1-1.center); 
\draw [thick,dashed] (M-2-1.center) -- (M-3-2.center); 
\draw [thick,dashed] (M-3-2.center) -- (M-3-3.center); 
\draw [thick,dashed] (M-3-3.center) -- (M-1-3.center); 
\draw [thick,dashed] (M-1-3.center) -- (M-1-1.center); 

% inv Hess block
\draw [thick,dotted] (M-2-2.center) -- (M-5-2.center); 
%\draw [thick,dotted] (M-2-3.center) -- (M-4-5.center); 
\draw [thick,dotted] (M-5-5.center) -- (M-5-4.center); 
\draw [thick,dotted] (M-5-4.center) -- (M-5-2.center); 
\draw [thick] (M-2-2.center) -- (M-2-5.center); 
\draw [thick] (M-2-5.center) -- (M-5-5.center); 
\draw [thick,dotted] (M-2-2.center) -- (M-5-5.center); 

% Hess block
\draw [thick,dashed] (M-5-4.center) -- (M-4-4.center); 
\draw [thick,dashed] (M-4-4.center) -- (M-4-7.center); 
\draw [thick,dashed] (M-4-7.center) -- (M-7-7.center); 
\draw [thick,dashed] (M-7-7.center) -- (M-7-6.center); 
\draw [thick,dashed] (M-7-6.center) -- (M-5-4.center); 
 
\node[text width=3cm] at (3.9,0) {{$+$}};

\matrix (D) [matrix of nodes,left delimiter={[},right delimiter={]}]  at (5,0)
{$\times$&\\
	&$\times$\\
	&&$\xi_2$\\
	&&&$\xi_3$\\
	&&&&$\times$\\
	&&&&&$\times$\\
	&&&&&&$\times$\\
};

\end{tikzpicture}}
				\caption{Rational Hessenberg matrix $Z$ corresponding to the projection onto $\K(A,v;\Xi)$ from Example \ref{example:blocks}, with $\Xi = \{\infty, \xi_2,\xi_3,\infty, \infty, \infty \}$.}
				\label{fig:QRplusD}
			\end{figure}
	\end{example}

	The structure of the single-matrix representation can be explained through its link with the Hessenberg pencil representation discussed in Section \ref{sec:OrthExt_pair}, see Camps et al. \cite{CaMeVa19}.

	\subsection{Pencil representation}\label{sec:OrthExt_pair}
	A pencil representation of the projected matrix onto a rational Krylov subspace can be constructed via an Arnoldi iteration. Theorem \ref{theorem:RAI} provides the rational Arnoldi iteration in its most general form, i.e., expansion of the subspace is done with the operator $(\nu_k A-\mu_k I)^{-1} (\rho_k A-\eta_k I)$.
	
	\begin{theorem}[Rational Arnoldi iteration \cite{Ru84}]\label{theorem:RAI}
		Consider a matrix $A\in \mathbb{C}^{m\times m}$ and an orthonormal nested basis $V_n$ for a rational Krylov subspace $\K_n(A,v;\Xi_n)$, where $\Xi_n = \{ \xi_1, \xi_2,\dots,\xi_{n-1}\} \in \bar{\mathbb{C}}$. 
		The recurrence relation to obtain a basis $V_{n+1}$ for $\K_{n+1}(A,v;\Xi_{n+1})$, with $\Xi_{n+1} = \{\Xi_{n},\xi_n\}$,  in matrix form equals
		\begin{equation*}
		AV_{n+1} \underbar{$K$}_n = V_{n+1} \underbar{$H$}_n,
		\end{equation*}
		with $\underbar{$H$}_n$ and $\underbar{$K$}_n$ Hessenberg matrices of size $(n+1)\times n$. The ratio of their subdiagonal elements equals the poles of the rational Krylov subspace $\frac{(\underbar{$K$}_n)_{k+1,k}}{(\underbar{$H$}_n)_{k+1,k}} = \xi_k$, $k=1,2\dots, n$.
	\end{theorem}
	\begin{proof}
		Consider the formula for expanding the Krylov subspace $\K_k(A,v;\Xi)$ by multiplication with $(\nu_k A-\mu_k I)^{-1} (\rho_k A-\eta_k I)$, the subspace is invariant under the shift operator $(\rho_k A-\eta_k I)$. Afterwards orthogonalization with respect to all vectors in the current basis $V_k = \begin{bmatrix}
		v_1 & v_2 & \cdots & v_k
		\end{bmatrix}$ is done using $h_{ik}$, $i=1,\dots,k$ and normalization using $h_{k+1,k}$. This leads to a Gram-Schmidt orthogonalization procedure
		\begin{equation}\label{eq:expand_RKS}
		h_{k+1,k} v_{k+1} = (\nu_k A-\mu_k I)^{-1} (\rho_k A-\eta_k I) v_k-h_{1k} v_1 - \dots -h_{kk}v_k.
		\end{equation}
		Rewriting \eqref{eq:expand_RKS} reveals the $k$th column of matrices $\underbar{$H$}_k$ and $\underbar{$K$}_k$
		\begin{align*}
		(\nu_k A-\mu_k I) h_{k+1,k} v_{k+1} &= (\rho_k A-\eta_k I)v_k - (\nu_k A-\mu_k I) \sum_{i=1}^{k} h_{ik} v_i\\
		\nu_k A h_{k+1,k} v_{k+1} + \nu_k A \sum_{i=1}^{k} h_{ik}v_i - \rho_k A v_k &= -\eta_k v_k + \mu_k \sum_{i=1}^{k} h_{ik} v_i + \mu_k h_{k+1,k} v_{k+1}\\
		A\left( (\nu_k \sum_{i=1}^{k+1} h_{ik}v_i) - \rho_k v_k\right) &= \mu_k \left(\sum_{i=1}^{k+1}h_{ik} v_i\right) - \eta_k v_k\\
		A \nu_k \begin{bmatrix}
		v_1 & \dots &v_k & v_{k+1} 	\end{bmatrix} \begin{bmatrix}
		h_{1k} \\
		\vdots \\
		h_{kk} - \rho_k/\nu_k \\
		h_{k+1,k} 	\end{bmatrix} &= \mu_k \begin{bmatrix}
		v_1 & \dots &v_k & v_{k+1} 	\end{bmatrix} \begin{bmatrix}
		h_{1k} \\
		\vdots \\
		h_{kk} - \eta_k/\mu_k \\
		h_{k+1,k} 	\end{bmatrix}.
		\end{align*}
		The last equation reveals that the subdiagonal element ratio is $\frac{\mu_k h_{k+1,k}}{\nu_k h_{k+1,k}} = \frac{\mu_k}{\nu_k} =\xi_k$.
	\end{proof}
	From Theorem \ref{theorem:RAI} we obtain a Hessenberg pencil $(H_n,K_n)$, satisfying $Z_n=H_nK_n^{-1}$ with $K_n$ nonsingular, which represents the projection
	\begin{equation}\label{eq:EKS_proj_pair}
	V_n^HAV_nK_n=H_n.
	\end{equation}
	Such a Hessenberg pencil will be called proper if it has no subdiagonal elements $h_{i+1,i}$ and $k_{i+1,i}$ simultaneously zero.
	Theorem \ref{theorem:RAI} implies that $\underbar{$H$}_n$ and $\underbar{$K$}_n$ are linked, their subdiagonal ratios reveal the poles of the rational Krylov subspace from which they originate. These ratios are, however, invariant when $\underbar{$H$}_n$ and $\underbar{$K$}_n$ are both multiplied with an upper-triangular matrix $R$ from the right, illustrating that the Hessenberg pencil $(\underbar{$H$}_n,\underbar{$K$}_n)$ is not unique.

	An implicit Q-theorem for matrix-pencils $(\underbar{$H$}_n,\underbar{$K$}_n)$ exists, if the poles and starting vector are chosen and the structure of the matrices in this pencil is fixed. If the structure of the matrices is chosen to be Hessenberg, then the implicit Q-theorem can be found in the dissertation of Berljafa \cite{Be17}, the paper by Berljafa et al. \cite{BeGu15} and the paper of Camps et al. \cite{CaMeVa19}. Theorem \ref{theorem:Q-theorem} states this result and shows a one-to-one relation between Hessenberg pencils and rational Krylov subspaces, therefore manipulating poles in the pencil corresponds to manipulating the subspaces.
		\begin{theorem}[Rational implicit Q-theorem \cite{Be17,BeGu15,CaMeVa19}]\label{theorem:Q-theorem}
			Consider a decomposition of the form
			\begin{equation*}
			AV_{n+1} \underbar{$K$}_n = V_{n+1} \underbar{$H_n$}
			\end{equation*}
			with $(n+1)\times n$ Hessenberg matrices $\underbar{$H$}_n$ and $\underbar{$K$}_n$, poles $\xi_i = \frac{h_{i+1,i}}{k_{i+1,i}}$, for $i=1,\dots,n$ and $V_{n+1}$ an orthonormal nested basis for the rational Krylov subspace $\K_{n+1}(A,v;\Xi)$, where $v = V_{n+1}e_1$ the first column of $V_{n+1}$ and $\Xi = \{\xi_1,\xi_2,\dots,\xi_n \}$ the set of poles.\\
			Then the Hessenberg pencil $(\underbar{$H$}_n,\underbar{$K$}_n)$ and the orthonormal matrix $V_{n+1}$ are essentially uniquely determined by the starting vector $v$ and the poles $\Xi$.
		\end{theorem}
		%\begin{proof}
		%	See \cite{Be17}.
		%\end{proof}
		
	Note that Theorem \ref{theorem:Q-theorem} states the uniqueness of the Hessenberg pencil. The pencil can, however, also be represented using matrices with another structure than Hessenberg. 
	Since for a nonsingular matrix $C$ the pair $(H_n C,K_n C)$ also satisfies \eqref{eq:EKS_proj_pair}. 	
	Besides the Hessenberg pencil, another important representation is an inv-Hessenberg pencil. This representation is discussed in Section \ref{sec:inv_Hess} and is important for the derivation of the main result of this text provided in Section \ref{sec:BiExt}.

	\subsection{Inv-Hessenberg pencil}\label{sec:inv_Hess}
	An inv-Hessenberg pencil satisfying \eqref{eq:EKS_proj_pair} is constructed in this section.
	\begin{property}[Transfer through property \cite{Va11}]\label{prop:transfer_through}
		A shape of core transformations can be transferred through an upper-triangular matrix $R$ without altering the shape. 
	\end{property}

	\begin{example}\label{example:transfer_through}
		The equality $C_1C_3C_2C_4R=\widetilde{R}\widetilde{C}_1\widetilde{C}_3\widetilde{C}_2\widetilde{C}_4$ holds, where $\widetilde{R}$ is upper-triangular. The matrices involved will generally change (its elements), but the shape, i.e., the mutual ordering of the core transformations (and therefore the structure of the resulting matrix) remains the same.
	\end{example}
	
		\begin{lemma}[Turnover lemma \cite{VaVBMa07}, Lemma 9.38]\label{lemma:shift_through}
			Consider the product of three core transformations $G_{i-1} G_i \hat{G}_{i-1}$.
			Then there exists an equivalent representation $\Gamma_{i} \Gamma_{i-1} \hat{\Gamma}_i$
			\begin{align*}
			\begin{array}{ccc}
			\vspace{-0.07cm}
			\Rc &  & \Rc\\
			\vspace{-0.07cm}
			\rc & \Rc & \rc\\
			\vspace{-0.07cm}
			& \rc
			\end{array}	&=
			\begin{array}{ccc}
			\vspace{-0.07cm}
			& \Rc & \\
			\vspace{-0.07cm}
			\Rc & \rc & \Rc\\
			\vspace{-0.07cm}
			\rc& & \rc
			\end{array}	\\
			G_{i-1} G_{i} \hat{G}_{i-1} & = \Gamma_{i} \Gamma_{i-1} \hat{\Gamma}_i,
			\end{align*}
			with matrices defined as
			\begin{align*}
			G_{i-1} := \begin{bmatrix}
			c_{i-1} & s_{i-1}\\
			-s_{i-1} &c_{i-1}\\
			& & 1
			\end{bmatrix}\qquad
			G_{i} := \begin{bmatrix}
			1 \\
			&c_{i} & s_{i}\\
			&-s_{i} &c_{i}
			\end{bmatrix} \qquad
			\hat{G}_{i-1} := \begin{bmatrix}
			\hat{c}_{i-1} & \hat{s}_{i-1}\\
			-\hat{s}_{i-1} &\hat{c}_{i-1}\\
			& & 1
			\end{bmatrix}\\
			\Gamma_{i} := \begin{bmatrix}
			1 \\
			&\gamma_{i} & \sigma_{i}\\
			&-\sigma_{i} &\gamma_{i}
			\end{bmatrix} \qquad
			{\Gamma}_{i-1} := \begin{bmatrix}
			{\gamma}_{i-1} & \sigma_{i-1}\\
			-{\sigma}_{i-1} &{\gamma}_{i-1}\\
			& & 1
			\end{bmatrix} \qquad
			\hat{\Gamma}_{i} := \begin{bmatrix}
			1 \\
			&\hat{\gamma}_{i} & \hat{\sigma}_{i}\\
			&-\hat{\sigma}_{i} &\hat{\gamma}_{i}
			\end{bmatrix}.
			\end{align*}
		\end{lemma}
		
		For ease of notation Theorem \ref{theorem:hessenberg-like_pair} is stated and proved for a full reduction ($n=m$) but is valid for partial reductions as well.

		\begin{theorem}[Inv-Hessenberg pencil for projection onto rational Krylov subspaces]\label{theorem:hessenberg-like_pair}
			Let $V\in \mathbb{C}^{m\times m}$ be an orthonormal nested basis for a rational Krylov subspace $\K_m(A,v;\Xi)$, $A\in \Cmm$, $v \in \Cm$ and poles $\Xi$. Then the orthogonal projection onto this subspace can be represented as two inv-Hessenberg matrices $\Hi,\Ki$, i.e., they satisfy the equation $V^HAV\Ki = \Hi$.
		\end{theorem}
		\begin{proof}
			The existence of a Hessenberg pair $(H,K)$ satisfying $$V^H A V K=H$$ follows immediately from the Arnoldi iteration in Theorem \ref{theorem:RAI}.\\
			From Lemma \ref{lemma:QR_Hess} and Property \ref{prop:transfer_through} it follows that there exist upper-triangular matrices $R_H$ and $R_K$ and unitary matrices consisting of a descending pattern of core transformations $Q_H$ and $Q_K$ such that
			\begin{align*}
			V^HAVR_K Q_K = R_H Q_H.
			\end{align*}
			Let us write it in a manner such that the structures are clear.\\
			\begin{figure}[!ht]
				\centering
				{\begin{tikzpicture} 
$\hspace{-1.1cm}$
$V^HAV$ \hspace{1.25cm} \matrix (RK) [matrix of nodes] 
{$ $&$ $&$ $&$ $&$ $ &$ $&$ $&$ $&$ $\\
$ $&$ $&$ $&$ $&$ $ &$ $&$ $&$ $&$ $\\
$ $&$ $&$ $&$ $&$ $ &$ $&$ $&$ $&$ $\\
$ $&$ $&$ $&$ $&$ $ &$ $&$ $&$ $&$ $\\
$ $&$ $&$ $&$ $&$ $ &$ $&$ $&$ $&$ $\\
$ $&$ $&$ $&$ $&$ $ &$ $&$ $&$ $&$ $\\
$ $&$ $&$ $&$ $&$ $ &$ $&$ $&$ $&$ $\\
$ $&$ $&$ $&$ $&$ $ &$ $&$ $&$ $&$ $\\
$ $&$ $&$ $&$ $&$ $ &$ $&$ $&$ $&$ $\\
};  

\draw [thick] (RK-1-1.center) -- (RK-9-9.center); 
\draw [thick] (RK-1-1.center) -- (RK-1-9.center); 
\draw [thick] (RK-1-9.center) -- (RK-9-9.center); 

\draw (RK-4-5) node[above=-0.45cm,transform canvas={xshift=1em}] {$R_K$};
\node[label=center: {\resizebox{2.7cm}{!}{$\displaystyle
		\begin{array}{ccccc}
		\vspace{-0.07cm}
		\Rc \\
		\vspace{-0.07cm}
		\rc & \Rc\\
		\vspace{-0.07cm}
		&  \rc & \Rc\\
		\vspace{-0.3cm} 
		&	 & \hspace{0.007cm}\rc & & \\
		\vspace{-0.3cm}
		&	 &  & \ddots & \\
		\vspace{-0.07cm}
		&	 &     &  	    & \Rc\\
		&	 &     &  	    & \rc\\
		\end{array}$
}}]() at (2.4,0){};

\node[label=center: {$=$}]() at (4,0){};

\matrix (RH) [matrix of nodes]  at (5.5,0)
{$ $&$ $&$ $&$ $&$ $ &$ $&$ $&$ $&$ $\\
	$ $&$ $&$ $&$ $&$ $ &$ $&$ $&$ $&$ $\\
	$ $&$ $&$ $&$ $&$ $ &$ $&$ $&$ $&$ $\\
	$ $&$ $&$ $&$ $&$ $ &$ $&$ $&$ $&$ $\\
	$ $&$ $&$ $&$ $&$ $ &$ $&$ $&$ $&$ $\\
	$ $&$ $&$ $&$ $&$ $ &$ $&$ $&$ $&$ $\\
	$ $&$ $&$ $&$ $&$ $ &$ $&$ $&$ $&$ $\\
	$ $&$ $&$ $&$ $&$ $ &$ $&$ $&$ $&$ $\\
	$ $&$ $&$ $&$ $&$ $ &$ $&$ $&$ $&$ $\\
};  

\draw [thick] (RH-1-1.center) -- (RH-9-9.center); 
\draw [thick] (RH-1-1.center) -- (RH-1-9.center); 
\draw [thick] (RH-1-9.center) -- (RH-9-9.center); 

\draw (RH-4-5) node[above=-0.45cm,transform canvas={xshift=1em}] {$R_H$};
\node[label=center: {\resizebox{2.7cm}{!}{$\displaystyle
		\begin{array}{ccccc}
		\vspace{-0.07cm}
		\Rc \\
		\vspace{-0.07cm}
		\rc & \Rc\\
		\vspace{-0.07cm}
		&  \rc & \Rc\\
		\vspace{-0.3cm} 
		&	 & \hspace{0.007cm}\rc & & \\
		\vspace{-0.3cm}
		&	 &  & \ddots & \\
		\vspace{-0.07cm}
		&	 &     &  	    & \Rc\\
		&	 &     &  	    & \rc\\
		\end{array}$
}}]() at (8,0){};
\end{tikzpicture}}
			\end{figure}
		
			\noindent Multiply from the right with $Q_K^H$ annihilate the descending pattern of core transformations on the left-hand side.\\
			\begin{figure}[!ht]
				\centering
				{\begin{tikzpicture} 
$\hspace{-1.1cm}$
$V^HAV$ \hspace{1.25cm} \matrix (RK) [matrix of nodes] 
{$ $&$ $&$ $&$ $&$ $ &$ $&$ $&$ $&$ $\\
$ $&$ $&$ $&$ $&$ $ &$ $&$ $&$ $&$ $\\
$ $&$ $&$ $&$ $&$ $ &$ $&$ $&$ $&$ $\\
$ $&$ $&$ $&$ $&$ $ &$ $&$ $&$ $&$ $\\
$ $&$ $&$ $&$ $&$ $ &$ $&$ $&$ $&$ $\\
$ $&$ $&$ $&$ $&$ $ &$ $&$ $&$ $&$ $\\
$ $&$ $&$ $&$ $&$ $ &$ $&$ $&$ $&$ $\\
$ $&$ $&$ $&$ $&$ $ &$ $&$ $&$ $&$ $\\
$ $&$ $&$ $&$ $&$ $ &$ $&$ $&$ $&$ $\\
};  

\draw [thick] (RK-1-1.center) -- (RK-9-9.center); 
\draw [thick] (RK-1-1.center) -- (RK-1-9.center); 
\draw [thick] (RK-1-9.center) -- (RK-9-9.center); 

\draw (RK-4-5) node[above=-0.45cm,transform canvas={xshift=1em}] {$R_K$};

\node[label=center: {$=$}]() at (1.8,0){};

\matrix (RH) [matrix of nodes]  at (3.3,0)
{$ $&$ $&$ $&$ $&$ $ &$ $&$ $&$ $&$ $\\
	$ $&$ $&$ $&$ $&$ $ &$ $&$ $&$ $&$ $\\
	$ $&$ $&$ $&$ $&$ $ &$ $&$ $&$ $&$ $\\
	$ $&$ $&$ $&$ $&$ $ &$ $&$ $&$ $&$ $\\
	$ $&$ $&$ $&$ $&$ $ &$ $&$ $&$ $&$ $\\
	$ $&$ $&$ $&$ $&$ $ &$ $&$ $&$ $&$ $\\
	$ $&$ $&$ $&$ $&$ $ &$ $&$ $&$ $&$ $\\
	$ $&$ $&$ $&$ $&$ $ &$ $&$ $&$ $&$ $\\
	$ $&$ $&$ $&$ $&$ $ &$ $&$ $&$ $&$ $\\
};  

\draw [thick] (RH-1-1.center) -- (RH-9-9.center); 
\draw [thick] (RH-1-1.center) -- (RH-1-9.center); 
\draw [thick] (RH-1-9.center) -- (RH-9-9.center); 

\draw (RH-4-5) node[above=-0.45cm,transform canvas={xshift=1em}] {$R_H$};
\node[label=center: {\resizebox{2.7cm}{!}{$\displaystyle
		\begin{array}{ccccc}
		\vspace{-0.07cm}
		\Rc \\
		\vspace{-0.07cm}
		\rc & \Rc\\
		\vspace{-0.07cm}
		&  \rc & \Rc\\
		\vspace{-0.3cm} 
		&	 & \rc & & \\
		\vspace{-0.3cm}
		&	 &  & \ddots & \\
		\vspace{-0.07cm}
		&	 &     &  	    & \Rc\\
		&	 &     &  	    & \rc\\
		\end{array}$
}}]() at (5.8,0){};

\node[label=center: {\resizebox{2.7cm}{!}{$\displaystyle
		\begin{array}{ccccc}
		\vspace{-0.07cm}
		& & & & \Rc \\
		\vspace{-0.07cm}
		& &  & \Rc &\rc\\
		\vspace{-0.07cm}
		& & \Rc & \rc \\
		\vspace{-0.3cm}
		& & \rc \\
		\vspace{-0.3cm}
		 & \iddots \\
		 \vspace{-0.07cm}
		 \Rc & \\
		 \rc 
		\end{array}$
}}]() at (8.2,0){};

\end{tikzpicture}}
			\end{figure}
		
			\noindent Using the turnover operation repeatedly, see \cite{VaVBMa07} for details, it is possible to rearrange the core transformations to obtain another shape.\\
			\begin{figure}[!ht]
				\centering
				{\begin{tikzpicture} 
$AV$ \hspace{1.25cm} \matrix (RK) [matrix of nodes] 
{$ $&$ $&$ $&$ $&$ $ &$ $&$ $&$ $&$ $\\
$ $&$ $&$ $&$ $&$ $ &$ $&$ $&$ $&$ $\\
$ $&$ $&$ $&$ $&$ $ &$ $&$ $&$ $&$ $\\
$ $&$ $&$ $&$ $&$ $ &$ $&$ $&$ $&$ $\\
$ $&$ $&$ $&$ $&$ $ &$ $&$ $&$ $&$ $\\
$ $&$ $&$ $&$ $&$ $ &$ $&$ $&$ $&$ $\\
$ $&$ $&$ $&$ $&$ $ &$ $&$ $&$ $&$ $\\
$ $&$ $&$ $&$ $&$ $ &$ $&$ $&$ $&$ $\\
$ $&$ $&$ $&$ $&$ $ &$ $&$ $&$ $&$ $\\
};  

\draw [thick] (RK-1-1.center) -- (RK-9-9.center); 
\draw [thick] (RK-1-1.center) -- (RK-1-9.center); 
\draw [thick] (RK-1-9.center) -- (RK-9-9.center); 

\draw (RK-4-5) node[above=-0.45cm,transform canvas={xshift=1em}] {$R_K$};

\node[label=center: {$=$}]() at (1.8,0){};

\matrix (RH) [matrix of nodes]  at (3.3,0)
{$ $&$ $&$ $&$ $&$ $ &$ $&$ $&$ $&$ $\\
	$ $&$ $&$ $&$ $&$ $ &$ $&$ $&$ $&$ $\\
	$ $&$ $&$ $&$ $&$ $ &$ $&$ $&$ $&$ $\\
	$ $&$ $&$ $&$ $&$ $ &$ $&$ $&$ $&$ $\\
	$ $&$ $&$ $&$ $&$ $ &$ $&$ $&$ $&$ $\\
	$ $&$ $&$ $&$ $&$ $ &$ $&$ $&$ $&$ $\\
	$ $&$ $&$ $&$ $&$ $ &$ $&$ $&$ $&$ $\\
	$ $&$ $&$ $&$ $&$ $ &$ $&$ $&$ $&$ $\\
	$ $&$ $&$ $&$ $&$ $ &$ $&$ $&$ $&$ $\\
};  

\draw [thick] (RH-1-1.center) -- (RH-9-9.center); 
\draw [thick] (RH-1-1.center) -- (RH-1-9.center); 
\draw [thick] (RH-1-9.center) -- (RH-9-9.center); 

\draw (RH-4-5) node[above=-0.45cm,transform canvas={xshift=1em}] {$R_H$};
\node[label=center: {\resizebox{2.7cm}{!}{$\displaystyle
		\begin{array}{ccccc}
		\vspace{-0.07cm}
		\Rc \\
		\vspace{-0.07cm}
		\rc & \Rc\\
		\vspace{-0.07cm}
		&  \rc & \Rc\\
		\vspace{-0.3cm} 
		&	 & \rc & & \\
		\vspace{-0.3cm}
		&	 &  & \ddots & \\
		\vspace{-0.07cm}
		&	 &     &  	    & \Rc\\
		&	 &     &  	    & \rc\\
		\end{array}$
}}]() at (8.2,0){};

\node[label=center: {\resizebox{2.7cm}{!}{$\displaystyle
		\begin{array}{ccccc}
		\vspace{-0.07cm}
		& & & & \Rc \\
		\vspace{-0.07cm}
		& &  & \Rc &\rc\\
		\vspace{-0.07cm}
		& & \Rc & \rc \\
		\vspace{-0.3cm}
		& & \rc \\
		\vspace{-0.3cm}
		 & \iddots \\
		 \vspace{-0.07cm}
		 \Rc & \\
		 \rc 
		\end{array}$
}}]() at (5.8,0){};

\end{tikzpicture}}
			\end{figure}
		
			\noindent Multiplying from the right to annihilate the descending pattern of core transformations brings them back to the left-hand side.
			\begin{figure}[!ht]
				\centering
				{\begin{tikzpicture} 
$AV$ \hspace{1.25cm} \matrix (RK) [matrix of nodes] 
{$ $&$ $&$ $&$ $&$ $ &$ $&$ $&$ $&$ $\\
$ $&$ $&$ $&$ $&$ $ &$ $&$ $&$ $&$ $\\
$ $&$ $&$ $&$ $&$ $ &$ $&$ $&$ $&$ $\\
$ $&$ $&$ $&$ $&$ $ &$ $&$ $&$ $&$ $\\
$ $&$ $&$ $&$ $&$ $ &$ $&$ $&$ $&$ $\\
$ $&$ $&$ $&$ $&$ $ &$ $&$ $&$ $&$ $\\
$ $&$ $&$ $&$ $&$ $ &$ $&$ $&$ $&$ $\\
$ $&$ $&$ $&$ $&$ $ &$ $&$ $&$ $&$ $\\
$ $&$ $&$ $&$ $&$ $ &$ $&$ $&$ $&$ $\\
};  

\draw [thick] (RK-1-1.center) -- (RK-9-9.center); 
\draw [thick] (RK-1-1.center) -- (RK-1-9.center); 
\draw [thick] (RK-1-9.center) -- (RK-9-9.center); 

\draw (RK-4-5) node[above=-0.45cm,transform canvas={xshift=1em}] {$R_K$};
\node[label=center: {\resizebox{2.7cm}{!}{$\displaystyle
		\begin{array}{ccccc}
		\vspace{-0.07cm}
		& & & & \Rc \\
		\vspace{-0.07cm}
		& &  & \Rc &\rc\\
		\vspace{-0.07cm}
		& & \Rc & \rc \\
		\vspace{-0.3cm}
		& & \rc \\
		\vspace{-0.3cm}
		& \iddots \\
		\vspace{-0.07cm}
		\Rc & \\
		\rc 
		\end{array}$
}}]() at (2.4,0){};

\node[label=center: {$=$}]() at (4,0){};

\matrix (RH) [matrix of nodes]  at (5.5,0)
{$ $&$ $&$ $&$ $&$ $ &$ $&$ $&$ $&$ $\\
	$ $&$ $&$ $&$ $&$ $ &$ $&$ $&$ $&$ $\\
	$ $&$ $&$ $&$ $&$ $ &$ $&$ $&$ $&$ $\\
	$ $&$ $&$ $&$ $&$ $ &$ $&$ $&$ $&$ $\\
	$ $&$ $&$ $&$ $&$ $ &$ $&$ $&$ $&$ $\\
	$ $&$ $&$ $&$ $&$ $ &$ $&$ $&$ $&$ $\\
	$ $&$ $&$ $&$ $&$ $ &$ $&$ $&$ $&$ $\\
	$ $&$ $&$ $&$ $&$ $ &$ $&$ $&$ $&$ $\\
	$ $&$ $&$ $&$ $&$ $ &$ $&$ $&$ $&$ $\\
};  

\draw [thick] (RH-1-1.center) -- (RH-9-9.center); 
\draw [thick] (RH-1-1.center) -- (RH-1-9.center); 
\draw [thick] (RH-1-9.center) -- (RH-9-9.center); 

\draw (RH-4-5) node[above=-0.45cm,transform canvas={xshift=1em}] {$R_H$};
\node[label=center: {\resizebox{2.7cm}{!}{$\displaystyle
		\begin{array}{ccccc}
		\vspace{-0.07cm}
		& & & & \Rc \\
		\vspace{-0.07cm}
		& &  & \Rc &\rc\\
		\vspace{-0.07cm}
		& & \Rc & \rc \\
		\vspace{-0.3cm}
		& & \rc \\
		\vspace{-0.3cm}
		& \iddots \\
		\vspace{-0.07cm}
		\Rc & \\
		\rc 
		\end{array}$
}}]() at (8.0,0){};

\end{tikzpicture}}
			\end{figure}
			\newpage
			\noindent Using the notation from before, where dashed lines indicate low-rank structure.
			\begin{figure}[H]
			\centering
			{\begin{tikzpicture} 

\node[label=center: {}]() at (5,0){};
$\hspace{-1.1cm}$
$V^HAV$ \hspace{1.25cm} \matrix (RK) [matrix of nodes] 
{$ $&$ $&$ $&$ $&$ $ &$ $&$ $&$ $&$ $\\
$ $&$ $&$ $&$ $&$ $ &$ $&$ $&$ $&$ $\\
$ $&$ $&$ $&$ $&$ $ &$ $&$ $&$ $&$ $\\
$ $&$ $&$ $&$ $&$ $ &$ $&$ $&$ $&$ $\\
$ $&$ $&$ $&$ $&$ $ &$ $&$ $&$ $&$ $\\
$ $&$ $&$ $&$ $&$ $ &$ $&$ $&$ $&$ $\\
$ $&$ $&$ $&$ $&$ $ &$ $&$ $&$ $&$ $\\
$ $&$ $&$ $&$ $&$ $ &$ $&$ $&$ $&$ $\\
$ $&$ $&$ $&$ $&$ $ &$ $&$ $&$ $&$ $\\
};

\draw [thick] (RK-1-1.center) -- (RK-1-9.center); 
\draw [thick] (RK-1-9.center) -- (RK-9-9.center); 
\draw [thick, dashed] (RK-1-1.center) -- (RK-9-1.center); 
\draw [thick, dashed] (RK-9-1.center) -- (RK-9-9.center); 
\draw [thick, dashed] (RK-1-1.center) -- (RK-9-9.center); 
\draw (RK-4-5) node[above=-0.45cm,transform canvas={xshift=1em}] {$\Ki$};

\node[label=center: {$=$}]() at (1.7,0){};

\matrix (RH) [matrix of nodes]  at (3.5,0)
{$ $&$ $&$ $&$ $&$ $ &$ $&$ $&$ $&$ $\\
	$ $&$ $&$ $&$ $&$ $ &$ $&$ $&$ $&$ $\\
	$ $&$ $&$ $&$ $&$ $ &$ $&$ $&$ $&$ $\\
	$ $&$ $&$ $&$ $&$ $ &$ $&$ $&$ $&$ $\\
	$ $&$ $&$ $&$ $&$ $ &$ $&$ $&$ $&$ $\\
	$ $&$ $&$ $&$ $&$ $ &$ $&$ $&$ $&$ $\\
	$ $&$ $&$ $&$ $&$ $ &$ $&$ $&$ $&$ $\\
	$ $&$ $&$ $&$ $&$ $ &$ $&$ $&$ $&$ $\\
	$ $&$ $&$ $&$ $&$ $ &$ $&$ $&$ $&$ $\\
};  

\draw [thick] (RH-1-1.center) -- (RH-1-9.center); 
\draw [thick] (RH-1-9.center) -- (RH-9-9.center); 
\draw [thick, dashed] (RH-1-1.center) -- (RH-9-1.center); 
\draw [thick, dashed] (RH-9-1.center) -- (RH-9-9.center); 
\draw [thick, dashed] (RH-1-1.center) -- (RH-9-9.center);

\draw (RH-4-5) node[above=-0.45cm,transform canvas={xshift=1em}] {$\Hi$};

\end{tikzpicture}}
		\end{figure}		
			\noindent Hence, we have constructed an inv-Hessenberg pair $(\Hi,\Ki)$ which satisfies the matrix equality
			$$
			V^HAV\Ki=\Hi.\qquad \qedhere
			$$ 
		\end{proof}	
  \subsection{Connection to other factorizations} \label{sec:structure_orth_summary}
	Theorem \ref{theorem:hessenberg-like_pair}, Theorem \ref{theorem:RAI} and the result from Section \ref{sec:OrthExt_single} are summarized in Table \ref{table:structure_orthog}.
		\begin{table}[!ht]
			\centering
		%	\begin{adjustwidth}{-0.39in}{-0.2in}
		\begin{tabular}[t]{l|l}
			$V^HAV=Z$                & $V^HAVK=H$\\
			\hline
			 \begin{tikzpicture} 
\matrix (HV) [matrix of nodes]
{$ $&$ $ &$ $&$ $ &$ $\\
	$ $&$ $ &$ $&$ $ &$ $\\
	$ $&$ $ &$ $&$ $ &$ $\\
	$ $&$ $ &$ $&$ $ &$ $\\
	$ $&$ $ &$ $&$ $ &$ $\\
};  

\draw [dashed,thick] (HV-2-1.center) -- (HV-5-4.center); 
\draw [dashed,thick] (HV-2-1.center) -- (HV-5-1.center); 
\draw [dashed,thick] (HV-5-1.center) -- (HV-5-4.center); 
\draw [thick] (HV-1-1.center) -- (HV-1-5.center); 
\draw [thick] (HV-1-5.center) -- (HV-5-5.center); 
\draw [thick] (HV-2-1.center) -- (HV-1-1.center); 
\draw [thick] (HV-5-4.center) -- (HV-5-5.center); 
\end{tikzpicture} & \multicolumn{1}{l}{\makecell[l]{\begin{tikzpicture} 
\matrix (HV) [matrix of nodes]
{$ $&$ $ &$ $&$ $ &$ $\\
	$ $&$ $ &$ $&$ $ &$ $\\
	$ $&$ $ &$ $&$ $ &$ $\\
	$ $&$ $ &$ $&$ $ &$ $\\
	$ $&$ $ &$ $&$ $ &$ $\\
};  

\draw [thick] (HV-1-1.center) -- (HV-2-1.center); 
\draw [thick] (HV-2-1.center) -- (HV-5-4.center); 
\draw [thick] (HV-5-4.center) -- (HV-5-5.center); 
\draw [thick] (HV-1-1.center) -- (HV-1-5.center); 
\draw [thick] (HV-1-5.center) -- (HV-5-5.center); 

\matrix (HW) [matrix of nodes] at (1.3,0)
{$ $&$ $ &$ $&$ $ &$ $\\
	$ $&$ $ &$ $&$ $ &$ $\\
	$ $&$ $ &$ $&$ $ &$ $\\
	$ $&$ $ &$ $&$ $ &$ $\\
	$ $&$ $ &$ $&$ $ &$ $\\
};  

\draw [thick] (HW-1-1.center) -- (HW-2-1.center); 
\draw [thick] (HW-2-1.center) -- (HW-5-4.center); 
\draw [thick] (HW-5-4.center) -- (HW-5-5.center); 
\draw [thick] (HW-1-1.center) -- (HW-1-5.center); 
\draw [thick] (HW-1-5.center) -- (HW-5-5.center); 

\end{tikzpicture}\\ \begin{tikzpicture} 
\matrix (HV) [matrix of nodes]
{$ $&$ $ &$ $&$ $ &$ $\\
	$ $&$ $ &$ $&$ $ &$ $\\
	$ $&$ $ &$ $&$ $ &$ $\\
	$ $&$ $ &$ $&$ $ &$ $\\
	$ $&$ $ &$ $&$ $ &$ $\\
};  

\draw [dashed,thick] (HV-1-1.center) -- (HV-5-5.center); 
\draw [dashed,thick] (HV-1-1.center) -- (HV-5-1.center); 
\draw [dashed,thick] (HV-5-1.center) -- (HV-5-5.center); 
\draw [thick] (HV-1-1.center) -- (HV-1-5.center); 
\draw [thick] (HV-1-5.center) -- (HV-5-5.center); 

\matrix (HW) [matrix of nodes] at (1.3,0)
{$ $&$ $ &$ $&$ $ &$ $\\
	$ $&$ $ &$ $&$ $ &$ $\\
	$ $&$ $ &$ $&$ $ &$ $\\
	$ $&$ $ &$ $&$ $ &$ $\\
	$ $&$ $ &$ $&$ $ &$ $\\
};

\draw [dashed,thick] (HW-1-1.center) -- (HW-5-5.center); 
\draw [dashed,thick] (HW-1-1.center) -- (HW-5-1.center); 
\draw [dashed,thick] (HW-5-1.center) -- (HW-5-5.center); 
\draw [thick] (HW-1-1.center) -- (HW-1-5.center); 
\draw [thick] (HW-1-5.center) -- (HW-5-5.center); 

\end{tikzpicture}}} 
		\end{tabular}
		%	\end{adjustwidth}
		\caption{Summary of structures following from orthogonal projections onto a rational Krylov subspace with basis $V$.}
		\label{table:structure_orthog}
	\end{table}
	The structure of the single matrix representation also appeared in a paper by Mach et al. \cite{MaPrVa14} and the Hessenberg pencil representation is a classical result by Ruhe \cite{Ru84}. These results about structure provide the tools for deriving the tridiagonal pencil representation of oblique projection onto rational Krylov subspaces in Section \ref{sec:BiExt}.

	\section{Biorthogonal rational Krylov subspaces}\label{sec:BiExt}
	This section provides results concerning the possible structures of a biorthogonal projection onto rational Krylov subspaces \cite{La50}.\\
	Section \ref{sec:BiExt_single} provides results for the single-matrix representation. The main result of that section has been proven for extended Krylov subspaces by Mach, et al. \cite{MaVBVa14}.
	Section \ref{sec:BiExt_pair} provides novel results concerning structure of the pencil representation, elegantly generalizing the tridiagonal structure obtained by nonhermitian Lanczos.
	Section \ref{sec:structure_biorth_summary} provides an overview of the structures which are generalized by the results in this section. Based on the tridiagonal pencil a Lanczos iteration for rational Krylov subspaces is developed in Section \ref{sec:RatLan}.

	\subsection{Single-matrix representation}\label{sec:BiExt_single}
	The structure of a biorthogonal projection expressed as a single matrix $Z_n=W_n^HAV_n$ is given in Theorem \ref{theorem:BiExt_single}. Here $V_n$ and $W_n$ are  biorthogonal, i.e., $W_n^HV_n = I$, bases for two rational Krylov subspaces $\K(A,v;\Xi)$ and $\Ls(A^H,w;\varPhi)$, respectively, with $A\in \Cmm$, $v,w \in \Cm$ and $\Xi$, $\varPhi$ two (possibly) independent sets of poles.\\
	For the sake of readability the corresponding proof and all subsequent proofs are provided for a full reduction (i.e., $n=m$).\\
	The structure of $Z_n$ can be deduced using matrix factorizations rather than relying on orthogonality of the basisvectors for the subspaces, see, e.g., \cite{Wa93,JaRe11,Sc17}.\\
	
	\begin{theorem} [Structure of biorthogonal projection in single-matrix representation]\label{theorem:BiExt_single}
		Consider $A\in \mathbb{C}^{m\times m}$ and rational Krylov subspaces $\K(A,v;\Xi)$ and $\Ls(A^H,w;\varPhi)$ with biorthogonal bases $V,W \in \mathbb{C}^{m\times m}$, respectively. $\Xi$ and $\varPhi$ are two sets containing poles that are not in the spectrum of $A$. Under the assumption that no breakdowns occur, the biorthogonal projection
		$$
		W^H A V =Z,
		$$
		where $Z$ has the structure below its diagonal determined by the poles of $\K$ and the structure above its diagonal determined by poles of $\Ls$.
	\end{theorem}
	\begin{proof}
		A similar proof appears in \cite{MaVBVa14}. The proof is added for completeness.
		Consider the matrices $Z_V$ and $Z_W$
		\begin{align*}
		Z_V =\hat{V}^HA\hat{V}, \quad
		Z_W =\hat{W}^HA^H\hat{W},
		\end{align*}
		where $\hat{V}$ and $\hat{W}$ are orthogonal bases for the rational Krylov subspaces $\K(A,v;\Xi)$ and $\Ls(A^H,w;\varPhi)$, respectively, with $w^Hv=1$.
		In general for the orthogonal bases $\hat{W}^H\hat{V} \neq I$, taking the non-pivoted LR-decomposition of the matrix product $\hat{W}^H\hat{V}$ will allow us to construct biorthogonal bases $V$ and $W$. The non-pivoted LR-decomposition, consisting of a lower triangular $L$ and upper triangular $R$, will retain the nestedness of the bases $\hat{V}$ and $\hat{W}$ and will make the bases orthogonal to each other:
		\begin{align*}
		\hat{W}^H\hat{V} = LR\\
		\underbrace{L^{-1}\hat{W}^H}_{=:W^H}\underbrace{\hat{V}R^{-1}}_{=:V} = I\\
		W^HV =I.
		\end{align*}
		This decomposition exists if and only if $\hat{W}^H\hat{V}$ is strongly nonsingular, otherwise it will break at the first singular principal minor. This break corresponds to a breakdown and is typical for biorthogonal methods. In this case the structural results hold up to the occurrence of the breakdown.
		The structure of $Z$ can be derived as follows. First consider
		\begin{align*}
		AV&=VZ\\
		A\underbrace{VR}_{\hat{V}} &= \underbrace{VR}_{\hat{V}} R^{-1} Z R\\
		A\hat{V}  &= \hat{V} \underbrace{R^{-1} Z R}_{Z_V}
		\end{align*}
		which provides the equality
		\begin{equation}\label{eq:single-matrix_ZV}
		Z = RZ_VR^{-1}.
		\end{equation}
		Second consider the relations
		\begin{align*}
		A^HW&=WZ^H\\
		A^H\underbrace{WL^H}_{\hat{W}} &= \underbrace{WL^H}_{\hat{W}} L^{-H} Z^H L^H\\
		A^H\hat{W}  &= \hat{W} \underbrace{L^{-H} Z^H L^H}_{Z_W}
		\end{align*}
		which provides the equality
		\begin{equation}\label{eq:single-matrix_ZW}
		Z^H = L^H Z_W L^{-H}.
		\end{equation}
		Multiplication with an upper-triangular matrix preserves the structure in the lower triangular part. Hence, the structure of $Z$ is the same as the structure of $Z_V$ for the lower triangular part, following from \eqref{eq:single-matrix_ZV}. The structure of the upper triangular part of $Z$ is the same as the structure of the lower triangular part of $Z_W$, following from \eqref{eq:single-matrix_ZW}.
	\end{proof}
	Theorem \ref{theorem:BiExt_single} is illustrated by Example \ref{ex:rat_single} for extended Krylov subspaces.
	\begin{example} \label{ex:rat_single} Consider $A\in \mathbb{C}^{8\times 8}$ and extended Krylov subspaces
		\begin{align*}
		\K_8&=\spn\{v,Av,A^2v,A^3v,A^4v,A^{-1}v,A^5v,A^{-2}v\},\\
		\Ls_8&=\spn\{w,(A^H)^{-1}w,A^Hw,(A^H)^{-2}w,(A^H)^{-3}w,(A^H)^{-4}w,(A^H)^2w,(A^H)^3w\}.
		\end{align*}
		Orthogonal projection onto these subspaces results in matrices $Z_V$ for $\K_8$ and $Z_W$ for $\Ls_8$ and biorthogonal projection onto $\K_8$ and $\Ls_8$ results in $Z$. The structure of these matrices is shown in Figure \ref{fig:example3}. In case of rational Krylov subspaces we only have to include a diagonal matrix containing the poles. Note the extended Hessenberg structure for the orthogonal projections and the same structure appearing in the biorthogonal projection, but now below as well as above the diagonal. Black lines are added to highlight the structure.

		\begin{figure}[!ht]
			\centering
			\begin{subfigure}[b]{0.25\textwidth}
				\resizebox{3.1cm}{2.5cm}{	\begin{tikzpicture} 
 \matrix (M) [matrix of nodes,left delimiter={[},right delimiter={]}] 
{$\times$&$\times$&$\times$&$\times$&$\times$&$\times$&$\times$&$\times$\\
$\times$&$\times$&$\times$&$\times$&$\times$&$\times$&$\times$&$\times$\\
&$\times$&$\times$&$\times$&$\times$&$\times$&$\times$&$\times$\\
&&$\times$&$\times$&$\times$&$\times$&$\times$&$\times$\\
&&&$\times$&$\times$&$\times$&$\times$&$\times$\\
&&&&$\times$&$\times$&$\times$&$\times$\\
&&&&$\times$&$\times$&$\times$&$\times$\\
&&&&&&$\times$&$\times$\\
};
\matrix (M) [matrix of nodes,left delimiter={[},right delimiter={]}] 
{$\times$&$\times$&$\times$&$\times$&$\times$&$\times$&$\times$&$\times$\\
	$\times$&$\times$&$\times$&$\times$&$\times$&$\times$&$\times$&$\times$\\
	&$\times$&$\times$&$\times$&$\times$&$\times$&$\times$&$\times$\\
	&&$\times$&$\times$&$\times$&$\times$&$\times$&$\times$\\
	&&&$\times$&$\times$&$\times$&$\times$&$\times$\\
	&&&&$\times$&$\times$&$\times$&$\times$\\
	&&&&$\times$&$\times$&$\times$&$\times$\\
	&&&&&&$\times$&$\times$\\
};
 \draw [line width=0.25mm] (M-2-1.center) -- (M-3-2.center); 
\draw [line width=0.25mm] (M-3-2.center) -- (M-4-3.center); 
\draw [line width=0.25mm] (M-4-3.center) -- (M-5-4.center); 
\draw [line width=0.25mm] (M-5-4.center) -- (M-6-5.center); 
\draw [line width=0.25mm,dashed] (M-6-5.center) -- (M-7-5.center); 
\draw [line width=0.25mm,dashed] (M-7-5.center) -- (M-7-6.center); 
\draw [line width=0.25mm] (M-7-6.center) -- (M-8-7.center); 
\end{tikzpicture}}
				\caption{$\hat{V}^HA\hat{V} = Z_V$}
			\end{subfigure}
			\hspace{0.4cm}
			\begin{subfigure}[b]{0.25\textwidth}
				\resizebox{3.1cm}{2.5cm}{	\begin{tikzpicture} 
 \matrix (M) [matrix of nodes,left delimiter={[},right delimiter={]}] 
{$\times$&$\times$&$\times$&$\times$&$\times$&$\times$&$\times$&$\times$\\
$\times$&$\times$&$\times$&$\times$&$\times$&$\times$&$\times$&$\times$\\
$\times$&$\times$&$\times$&$\times$&$\times$&$\times$&$\times$&$\times$\\
&&$\times$&$\times$&$\times$&$\times$&$\times$&$\times$\\
&&$\times$&$\times$&$\times$&$\times$&$\times$&$\times$\\
&&$\times$&$\times$&$\times$&$\times$&$\times$&$\times$\\
&&$\times$&$\times$&$\times$&$\times$&$\times$&$\times$\\
&&&&&&$\times$&$\times$\\
};
 \matrix (M) [matrix of nodes,left delimiter={[},right delimiter={]}] 
{$\times$&$\times$&$\times$&$\times$&$\times$&$\times$&$\times$&$\times$\\
	$\times$&$\times$&$\times$&$\times$&$\times$&$\times$&$\times$&$\times$\\
	$\times$&$\times$&$\times$&$\times$&$\times$&$\times$&$\times$&$\times$\\
	&&$\times$&$\times$&$\times$&$\times$&$\times$&$\times$\\
	&&$\times$&$\times$&$\times$&$\times$&$\times$&$\times$\\
	&&$\times$&$\times$&$\times$&$\times$&$\times$&$\times$\\
	&&$\times$&$\times$&$\times$&$\times$&$\times$&$\times$\\
	&&&&&&$\times$&$\times$\\
};
 \draw [line width=0.25mm,dashed] (M-2-1.center) -- (M-3-1.center); 
\draw [line width=0.25mm,dashed] (M-3-1.center) -- (M-3-2.center); 
\draw [line width=0.25mm] (M-3-2.center) -- (M-4-3.center); 
\draw [line width=0.25mm,dashed] (M-4-3.center) -- (M-7-3.center); 
\draw [line width=0.25mm,dashed] (M-7-3.center) -- (M-7-6.center); 
\draw [line width=0.25mm] (M-7-6.center) -- (M-8-7.center); 
\end{tikzpicture}}
				\caption{$\hat{W}^HA^H\hat{W} = Z_W$}
			\end{subfigure}
			\hspace{0.4cm}
			\begin{subfigure}[b]{0.25\textwidth}
				\resizebox{3.1cm}{2.5cm}{	\begin{tikzpicture} 
 \matrix (M) [matrix of nodes,left delimiter={[},right delimiter={]}] 
{$\times$&$\times$&$\times$&&&&&\\
$\times$&$\times$&$\times$&&&&&\\
&$\times$&$\times$&$\times$&$\times$&$\times$&$\times$&\\
&&$\times$&$\times$&$\times$&$\times$&$\times$&\\
&&&$\times$&$\times$&$\times$&$\times$&\\
&&&&$\times$&$\times$&$\times$&\\
&&&&$\times$&$\times$&$\times$&$\times$\\
&&&&&&$\times$&$\times$\\
};
\matrix (M) [matrix of nodes,left delimiter={[},right delimiter={]}] 
{$\times$&$\times$&$\times$&&&&&\\
	$\times$&$\times$&$\times$&&&&&\\
	&$\times$&$\times$&$\times$&$\times$&$\times$&$\times$&\\
	&&$\times$&$\times$&$\times$&$\times$&$\times$&\\
	&&&$\times$&$\times$&$\times$&$\times$&\\
	&&&&$\times$&$\times$&$\times$&\\
	&&&&$\times$&$\times$&$\times$&$\times$\\
	&&&&&&$\times$&$\times$\\
};
 \draw [line width=0.25mm] (M-2-1.center) -- (M-3-2.center); 
\draw [line width=0.25mm] (M-3-2.center) -- (M-4-3.center); 
\draw [line width=0.25mm] (M-4-3.center) -- (M-5-4.center); 
\draw [line width=0.25mm] (M-5-4.center) -- (M-6-5.center); 
\draw [line width=0.25mm,dashed] (M-6-5.center) -- (M-7-5.center); 
\draw [line width=0.25mm,dashed] (M-7-5.center) -- (M-7-6.center); 
\draw [line width=0.25mm] (M-7-6.center) -- (M-8-7.center); 
\draw [line width=0.25mm,dashed] (M-1-2.center) -- (M-1-3.center); 
\draw [line width=0.25mm,dashed] (M-1-3.center) -- (M-2-3.center); 
\draw [line width=0.25mm] (M-2-3.center) -- (M-3-4.center); 
\draw [line width=0.25mm,dashed] (M-3-4.center) -- (M-3-7.center); 
\draw [line width=0.25mm,dashed] (M-3-7.center) -- (M-6-7.center); 
\draw [line width=0.25mm] (M-6-7.center) -- (M-7-8.center); 
\end{tikzpicture}}
				\caption{$W^HAV = Z$}
				\label{fig:biext_Z_single}
			\end{subfigure}
			\caption{Structure of orthogonal and biorthogonal projection onto the extended Krylov subspaces $\K_8$ and $\Ls_8$ in single-matrix representation, see Example \ref{ex:rat_single}.}
			\label{fig:example3}
		\end{figure}
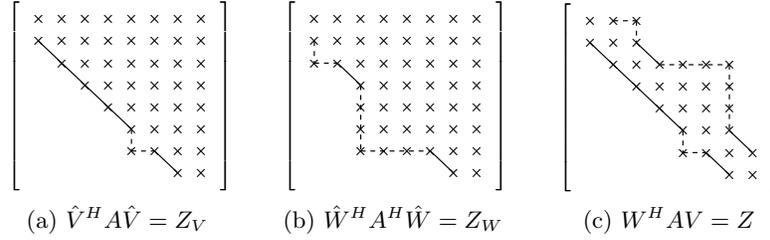
	\end{example}
	The following lemma provides a result for unitary matrices stated by Bunse-Gerstner and Fa{\ss}bender. \cite{BuFa97}, Stewart \cite{St07}, and several others. It follows easily from Theorem \ref{theorem:BiExt_single} and illustrates how the theorem can be used as a general framework to derive structures arising from projection onto Krylov subspaces.
	\begin{lemma}\label{lemma:unitary_single}
			Consider a unitary matrix $U\in \Cmm$, some starting vector $v\in \Cm$ and an orthogonal nested basis $V$ for the standard Krylov subspace $\K_m(U,v)$. Under the assumption that no breakdown occurs, the projection $Z = V^H U V$ has Hessenberg structure below and inv-Hessenberg structure above the diagonal.
	\end{lemma}
	\begin{proof}
		Consider a unitary matrix $U$, $U^{-1}=U^H$ and rational (more precisely extended) Krylov subspaces
			\begin{align*}
			\K(U,v;\Xi = \{\infty,\infty,\dots\}),\\
			\K(U^H,v;\varPhi = \{0,0,\dots\}),
			\end{align*}
		with respective orthogonal bases $\hat{V}$ and $\hat{W}$. Since $U^{-1}=U^H$, $\K(U,v;\Xi) = \K(U^H,v;\varPhi) = \mathcal{K}(U,v)$ and therefore $\hat{V}=\hat{W}=:V$ implying $\hat{V}^H\hat{W}=V^HV = I$. Hence, they are simultaneously orthogonal and biorthogonal bases.\\
		Using the knowledge from Section \ref{sec:OrthExt_single} it is clear that the structure of
		\begin{align*}
		Z_V &=\hat{V}^HA\hat{V},\\
		Z_W &= \hat{W}^HA^H\hat{W},
		\end{align*}
		is Hessenberg and inv-Hessenberg, respectively.
		Theorem \ref{theorem:BiExt_single} then states that $Z=\hat{W}^HA\hat{V}=V^HAV$ has Hessenberg structure in its lower triangular part and inv-Hessenberg structure in its upper triangular part.
	\end{proof}
	
	Retrieving the poles of both spaces $\K$ and $\Ls$ from the single-matrix representation $Z$ is possible but rather technical, especially in the parts where the matrix is not of Hessenberg form. Next section discusses the pencil representation which allows a more elegant retrieval of the poles in case a tridiagonal pencil is used.
	\subsection{Pencil representation}\label{sec:BiExt_pair}
	The main contribution of this text is the general pencil structure given in Theorem \ref{theorem:struct_BiRat_pair}. A specific instance is a tridiagonal pencil, which is formulated in Lemma \ref{theorem:tridiag_pair}. As a consequence of the tridiagonal pencil a six-term recurrence relation can be derived for the biorthogonal bases for rational Krylov subspaces in Section \ref{sec:RatLan}.
		\begin{theorem}[Pencil structure of biorthogonal projection onto rational Krylov subspaces]\label{theorem:struct_BiRat_pair}
			Consider $A \in  \Cmm$, two vectors $v,w \in \Cm$ and two rational Krylov subspaces $\K(A,v;\Xi)$ and $\Ls(A^H,w; \varPsi)$, where $\Xi$ and $\varPsi$ are two sets of poles (not in the spectrum of $A$).
			Let $\hat{V}, \hat{W}\in \Cmm$ be orthogonal nested bases and $V,W\in \Cmm$ biorthogonal nested bases for $\K$ and $\Ls$, respectively.
			Then there exist $H,K,H_V,K_V,H_W,K_W$ such that
			\begin{align*}
			\hat{V}^HA \hat{V} K_V &= H_V\\
			\hat{W}^HA^H \hat{W}K_W &= H_W\\
			W^H A V K &= H
			\end{align*}
			and the structure of the pencil $(H,K)$ is related to the structure of the pencils $(H_V,K_V)$ and $(H_W,K_W)$.
			Inverted structure is short for writing that a Hessenberg block becomes an inv-hessenberg block and vice versa, then the structure can be related as
			\begin{itemize}
				\item $H$ has the same structure below its diagonal as $H_V$ and above its diagonal the inverted structure of $K_W$
				\item $K$ has the same structure below its diagonal as $K_V$ and above its diagonal the inverted structure of $H_W$.
			\end{itemize}
			%	\begin{figure}[!ht]
			%		\centering
			%		{\footnotesize \input{figs/Struct_H_biorth.tikz}}
			%		{\footnotesize \input{figs/Struct_K_biorth.tikz}}
			%	\end{figure}
	\end{theorem}
	
	\begin{proof}
		From the orthogonal bases $\hat{V}$ and $\hat{W}$, the biorthogonal bases $V$ and $W$ can be constructed as in Theorem \ref{theorem:BiExt_single}, i.e., $V:=\hat{V}R^{-1}$ and $W^H:=L^{-1}\hat{W}^H$.
		Substituting the expressions for the biorthogonal bases in the equation of the orthogonal projection \eqref{eq:EKS_proj_pair} provides
		\begin{align*}
		&\begin{cases}
		A\hat{V}K_V = \hat{V}H_V\\
		A^H\hat{W}K_W = \hat{W}H_W
		\end{cases}\\
		\Leftrightarrow&\begin{cases}
		A\hat{V}R^{-1}RK_V = \hat{V}R^{-1}RH_V\\
		A^H\hat{W}L^{-H}L^H K_W = \hat{W}L^{-H}L^H H_W
		\end{cases}\\
		\Leftrightarrow&\begin{cases}
		A V R K_V = V R H_V\\
		A^H W L^H K_W = W L^H H_W
		\end{cases}\\
		\Leftrightarrow& \begin{cases}
		W^H A V R K_V=R H_V\\
		V^H A^H W L^H K_W = L^H H_W
		\end{cases}.
		\end{align*}
		Taking the Hermitian conjugate of the second equation and rewriting it reveals the connection between the matrices at play
		\begin{equation*}
		\begin{cases}
		W^H A V R K_V=R H_V\\
		W^H A V L^{-1} H^{-H}_W = L^{-1} K^{-H}_W
		\end{cases}.
		\end{equation*}
		Since these expressions are only unique up to right multiplication with a nonsingular matrix $B$, we get
		\begin{align*}
		RK_V B = L^{-1} H_W^{-H} \\
		RH_V  B =  L^{-1} K_W^{-H}.
		\end{align*}
		To obtain a particular choice for the structure of $H$ and $K$ it suffices to represent $B$ in its $RL$-decomposition (assuming it exists), where $R$ is an upper-triangular matrix and $L$ a lower-triangular matrix
		\begin{align*}
		&\begin{cases}
		RK_V B = L^{-1} H_W^{-H} \\
		RH_V  B =  L^{-1} K_W^{-H}
		\end{cases}\\
		\Leftrightarrow&\begin{cases}
		RK_V R_B L_B = L^{-1} H_W^{-H}\\
		RH_V  R_B L_B =  L^{-1} K_W^{-H}
		\end{cases}\\
		\Leftrightarrow&\begin{cases}
		RK_V R_B  = L^{-1} H_W^{-H}L_B^{-1}{\color{black}=:K}\\
		RH_V  R_B  =  L^{-1} K_W^{-H} L_B^{-1}{\color{black}=:H}
		\end{cases}.
		\end{align*}
		For the remainder of this proof $H$ and $K$ are defined as in the last equation. Other choices are possible because of the non-uniqueness of the pencil representation.
		Since $R$ and $R_B$ are upper-triangular matrices, they preserve the structure in the lower triangular part. This means that $K$ and $K_V$ have the same lower triangular structure and so do $H$ and $H_V$. On the other hand $K$ shares its upper triangular structure with $H_W^{-H}$ and $H$ with $K_W^{-H}$, since $L$ and $L_B$ are lower-triangular matrices. $\mkern 480mu$
	\end{proof}
	Starting from Theorem \ref{theorem:struct_BiRat_pair} it is straightforward to prove the following lemma.
		\begin{lemma}[Tridiagonal pencil for biorthogonal rational Krylov subspaces]\label{theorem:tridiag_pair}
			Consider some matrix $A\in \Cmm$ and vectors $v,w\in \Cm$. Let $V,W \in \Cmm$ be biorthogonal bases for rational Krylov subspaces $\K(A,v;\Xi)$ and $\Ls(A^H,w;\varPsi)$, where the poles are not in the spectrum of $A$. The equation
			\begin{equation}\label{eq:tridiag_pencil}
			W^HAVS=T
			\end{equation}
			representing the projection onto $\K$ and orthogonal to $\Ls$ is satisfied for a tridiagonal pencil $(T,S)$.
		\end{lemma}
		\begin{proof}
			If $(H_V,K_V)$ is chosen to be a Hessenberg pair (Theorem \ref{theorem:RAI}) and $(H_W,K_W)$ to be an inv-Hessenberg pair (Theorem \ref{theorem:hessenberg-like_pair}), then Theorem \ref{theorem:struct_BiRat_pair} guarantees that $(T,S)$ has tridiagonal structure.
	\end{proof}

	The pencil analogue to Lemma \ref{lemma:unitary_single} can be derived from Theorem \ref{theorem:struct_BiRat_pair}. This illustrates the ease with which structures in pencil form can be derived using this theorem. We stress that this result is only of theoretical use, not practical.
	\begin{lemma} \label{lemma:bidiag_pair}
			Consider a unitary matrix $U\in \Cmm$, some starting vector $v\in \Cm$ and an orthogonal nested basis $V$ for the standard Krylov subspace $\K(U,v)$. Under the assumption that no breakdown occurs, the equation $V^H U V K = H$ is satisfied for a proper lower-bidiagonal and upper-bidiagonal pencil $(H,K)$.
	\end{lemma}
	\begin{proof} Consider a unitary matrix $U$, $U^{-1}=U^H$ and the same subspaces $\K(U,v;\Xi)$ and $\K(U^H,v;\varPhi)$ as in the proof of Lemma \ref{lemma:unitary_single}, with respective orthogonal bases $\hat{V}$ and $\hat{W}$, note that $\hat{V}=\hat{W}=:V$. \\
		The pencil representation of orthogonal projections onto these subspaces are the following
		\begin{align*}
		\hat{V}^HU\hat{V}K_V = H_V,\\
		\hat{W}^HU^H\hat{W}K_W = H_W.
		\end{align*}
		For $(H_V,K_V)$, consider the standard case: $K_V$ is upper-triangular and $H_V$ is of Hessenberg form.	For $(H_W,K_W)$, choose $H_W$ to be of inv-Hessenberg form and $K_W$ to be upper-triangular.
		Then following from Theorem \ref{theorem:struct_BiRat_pair} the structure of $(H,K)$ is a lower-bidiagonal and upper-bidiagonal pencil.
	\end{proof}

	Lemma \ref{lemma:bidiag_pair} together with Lemma \ref{lemma:unitary_single} shows that a unitary Hessenberg matrix $Z$ can be factorized as the product of a lower-bidiagonal matrix $H$ and the inverse of an upper-bidiagonal matrix $K$ \cite{VaVBMa07}, using the notation from the lemmas.

  \subsection{Connection to other factorizations} \label{sec:structure_biorth_summary}
The results from this section generalize many well-known results such as the nonhermitian Lanczos iteration \cite{La50}, the Hermitian rational Lanczos iteration \cite{DeBu07} , AGR-or CMV-factorization \cite{Wa93} and more recent results by Jagels and Reichel \cite{JaRe11}, Schweitzer \cite{Sc17}, and others. Table \ref{table:structure} visualizes the structures of biorthogonal projection onto rational Krylov subspaces. The main contribution of this paper are the entries on the right, biorthogonal projection onto rational Krylov subspaces for pencil representation. Through Theorem \ref{theorem:struct_BiRat_pair} it is possible to relate the structures appearing in Table \ref{table:structure_orthog} to those in Table \ref{table:structure}.
	\begin{table}[!ht]
		\centering
		%	\begin{adjustwidth}{-0.39in}{-0.2in}
		\begin{tabular}[t]{l|ll}
			 $W^HAV=Z$                 & \multicolumn{2}{c}{$W^HAVK=H$}                \\
			 \hline
			 \begin{tikzpicture} 
\matrix (HV) [matrix of nodes]
{$ $&$ $ &$ $&$ $ &$ $\\
	$ $&$ $ &$ $&$ $ &$ $\\
	$ $&$ $ &$ $&$ $ &$ $\\
	$ $&$ $ &$ $&$ $ &$ $\\
	$ $&$ $ &$ $&$ $ &$ $\\
};  

\draw [dashed,thick] (HV-2-1.center) -- (HV-5-4.center); 
\draw [dashed,thick] (HV-1-2.center) -- (HV-4-5.center); 
\draw [dashed,thick] (HV-2-1.center) -- (HV-5-1.center); 
\draw [dashed,thick] (HV-5-1.center) -- (HV-5-4.center); 
\draw [dashed,thick] (HV-1-2.center) -- (HV-1-5.center); 
\draw [dashed,thick] (HV-1-5.center) -- (HV-4-5.center); 
\draw [thick] (HV-2-1.center) -- (HV-1-1.center); 
\draw [thick] (HV-1-2.center) -- (HV-1-1.center); 
\draw [thick] (HV-5-4.center) -- (HV-5-5.center); 
\draw [thick] (HV-4-5.center) -- (HV-5-5.center); 
\end{tikzpicture} & %\input{figs/brat2.tikz}
			  {\makecell[l]{\begin{tikzpicture} 
\matrix (HV) [matrix of nodes]
{$ $&$ $&$ $&$ $\\
	$ $&$ $&$ $&$ $\\
	$ $&$ $&$ $&$ $\\
	$ $&$ $&$ $&$ $\\
};  

\draw [thick] (HV-2-1.center) -- (HV-4-3.center); 
\draw [thick] (HV-1-2.center) -- (HV-3-4.center); 
\draw [thick] (HV-1-2.center) -- (HV-1-1.center); 
\draw [thick] (HV-2-1.center) -- (HV-1-1.center); 
\draw [thick] (HV-4-3.center) -- (HV-4-4.center); 
\draw [thick] (HV-3-4.center) -- (HV-4-4.center); 

\matrix (HW) [matrix of nodes] at (1.3,0)
{$ $&$ $&$ $&$ $\\
	$ $&$ $&$ $&$ $\\
	$ $&$ $&$ $&$ $\\
	$ $&$ $&$ $&$ $\\
};  

\draw [thick] (HW-2-1.center) -- (HW-4-3.center); 
\draw [thick] (HW-1-2.center) -- (HW-3-4.center); 
\draw [thick] (HW-1-2.center) -- (HW-1-1.center); 
\draw [thick] (HW-2-1.center) -- (HW-1-1.center); 
\draw [thick] (HW-4-3.center) -- (HW-4-4.center); 
\draw [thick] (HW-3-4.center) -- (HW-4-4.center);

\end{tikzpicture} \\	\begin{tikzpicture} 
\matrix (HV) [matrix of nodes]
{$ $&$ $&$ $&$ $\\
	$ $&$ $&$ $&$ $\\
	$ $&$ $&$ $&$ $\\
	$ $&$ $&$ $&$ $\\
};  

\draw [thick] (HV-2-1.center) -- (HV-4-3.center); 
%\draw [thick] (HV-1-2.center) -- (HV-3-4.center); 
\draw [thick] (HV-1-2.center) -- (HV-1-1.center); 
\draw [thick] (HV-2-1.center) -- (HV-1-1.center); 
\draw [thick] (HV-4-3.center) -- (HV-4-4.center); 
%\draw [thick] (HV-3-4.center) -- (HV-4-4.center); 
\draw [dashed,thick] (HV-1-1.center) -- (HV-4-4.center); 
\draw [dashed,thick] (HV-1-1.center) -- (HV-1-4.center); 
\draw [dashed,thick] (HV-1-4.center) -- (HV-4-4.center); 

\matrix (HW) [matrix of nodes] at (1.3,0)
{$ $&$ $&$ $&$ $\\
	$ $&$ $&$ $&$ $\\
	$ $&$ $&$ $&$ $\\
	$ $&$ $&$ $&$ $\\
};  

\draw [thick] (HW-2-1.center) -- (HW-4-3.center); 
%\draw [thick] (HW-1-2.center) -- (HW-3-4.center); 
\draw [thick] (HW-1-2.center) -- (HW-1-1.center); 
\draw [thick] (HW-2-1.center) -- (HW-1-1.center); 
\draw [thick] (HW-4-3.center) -- (HW-4-4.center); 
%\draw [thick] (HW-3-4.center) -- (HW-4-4.center); 
\draw [dashed,thick] (HW-1-1.center) -- (HW-4-4.center); 
\draw [dashed,thick] (HW-1-1.center) -- (HW-1-4.center); 
\draw [dashed,thick] (HW-1-4.center) -- (HW-4-4.center);

\end{tikzpicture}}} &  {\makecell[l]{ \begin{tikzpicture} 
\matrix (HV) [matrix of nodes]
{$ $&$ $&$ $&$ $\\
	$ $&$ $&$ $&$ $\\
	$ $&$ $&$ $&$ $\\
	$ $&$ $&$ $&$ $\\
};  

\draw [thick] (HV-1-2.center) -- (HV-3-4.center); 
%\draw [thick] (HV-1-2.center) -- (HV-3-4.center); 
\draw [thick] (HV-2-1.center) -- (HV-1-1.center); 
\draw [thick] (HV-1-2.center) -- (HV-1-1.center); 
\draw [thick] (HV-3-4.center) -- (HV-4-4.center); 
%\draw [thick] (HV-3-4.center) -- (HV-4-4.center); 
\draw [dashed,thick] (HV-1-1.center) -- (HV-4-4.center); 
\draw [dashed,thick] (HV-1-1.center) -- (HV-4-1.center); 
\draw [dashed,thick] (HV-4-1.center) -- (HV-4-4.center); 

\matrix (HW) [matrix of nodes] at (1.3,0)
{$ $&$ $&$ $&$ $\\
	$ $&$ $&$ $&$ $\\
	$ $&$ $&$ $&$ $\\
	$ $&$ $&$ $&$ $\\
};  

\draw [thick] (HW-1-2.center) -- (HW-3-4.center); 
%\draw [thick] (HW-1-2.center) -- (HW-3-4.center); 
\draw [thick] (HW-2-1.center) -- (HW-1-1.center); 
\draw [thick] (HW-1-2.center) -- (HW-1-1.center); 
\draw [thick] (HW-3-4.center) -- (HW-4-4.center); 
%\draw [thick] (HW-3-4.center) -- (HW-4-4.center); 
\draw [dashed,thick] (HW-1-1.center) -- (HW-4-4.center); 
\draw [dashed,thick] (HW-1-1.center) -- (HW-4-1.center); 
\draw [dashed,thick] (HW-4-1.center) -- (HW-4-4.center);

\end{tikzpicture}  \\	\begin{tikzpicture} 
\matrix (HV) [matrix of nodes]
{$ $&$ $&$ $&$ $\\
	$ $&$ $&$ $&$ $\\
	$ $&$ $&$ $&$ $\\
	$ $&$ $&$ $&$ $\\
};  

\draw [dashed,thick] (HV-1-1.center) -- (HV-4-4.center); 
\draw [dashed,thick] (HV-1-1.center) -- (HV-1-4.center); 
\draw [dashed,thick] (HV-1-4.center) -- (HV-4-4.center); 
\draw [thick] (HV-1-1.center) -- (HV-4-4.center); 
%\draw [thick] (HV-3-4.center) -- (HV-4-4.center); 
\draw [dashed,thick] (HV-1-1.center) -- (HV-4-4.center); 
\draw [dashed,thick] (HV-1-1.center) -- (HV-4-1.center); 
\draw [dashed,thick] (HV-4-1.center) -- (HV-4-4.center); 

\matrix (HW) [matrix of nodes] at (1.3,0)
{$ $&$ $&$ $&$ $\\
	$ $&$ $&$ $&$ $\\
	$ $&$ $&$ $&$ $\\
	$ $&$ $&$ $&$ $\\
};

\draw [dashed,thick] (HW-1-1.center) -- (HW-4-4.center); 
\draw [dashed,thick] (HW-1-1.center) -- (HW-1-4.center); 
\draw [dashed,thick] (HW-1-4.center) -- (HW-4-4.center); 
\draw [thick] (HW-1-1.center) -- (HW-4-4.center); 
%\draw [thick] (HW-3-4.center) -- (HW-4-4.center); 
\draw [dashed,thick] (HW-1-1.center) -- (HW-4-4.center); 
\draw [dashed,thick] (HW-1-1.center) -- (HW-4-1.center); 
\draw [dashed,thick] (HW-4-1.center) -- (HW-4-4.center);

\end{tikzpicture}  }}
		\end{tabular}
		%	\end{adjustwidth}
		\caption{Summary of pencil structures occurring for biorthogonal projections onto rational Krylov subspaces with bases $V$ and $W$.}
		\label{table:structure}
	\end{table}
	\newpage
\section{Rational Lanczos iteration}\label{sec:RatLan}
		A Lanczos-type iteration which constructs biorthogonal bases for rational Krylov subspaces is tested in this section. Section \ref{sec:poles} states results concerning the appearance of the poles of rational Krylov subspaces as ratios in the tridiagonal pencil \eqref{eq:tridiag_pencil}. These results allow for the development of the Lanczos-type iteration, code implementing this iteration is included as an attachment.
		Some numerical results for this iteration are given in Section \ref{sec:numerics}, these serve as a proof of concept, we have not yet focused on numerical stability. 
		
		\subsection{Lanczos-type iteration}\label{sec:poles}
			For the construction of a Lanczos-type iteration we must know how the poles appear in the tridiagonal pencil \eqref{eq:tridiag_pencil}, Lemma \ref{lemma:subdiag} and Lemma \ref{lemma:superdiag} state this.
			\begin{lemma}\label{lemma:subdiag}
				Assume we have $W^H A V S = T$ as in Theorem \ref{theorem:struct_BiRat_pair}. The ratio of the subdiagonal elements of $(T,S)$ reveals the poles $\Xi = \{\xi_1,\xi_2,\dots,\xi_{m-1} \}$  of $\K(A,v;\Xi)$
				\begin{equation}
				\frac{T_{i+1,i}}{S_{i+1,i}} = \xi_i , \qquad i = 1,2,\dots,m-1.
				\end{equation}
			\end{lemma}
			\begin{proof} From Theorem \ref{theorem:RAI} it follows that the ratio of the subdiagonal elements of $(H_V,K_V)$ equals the poles
				$$
				\frac{(H_V)_{i+1,i}}{(K_V)_{i+1,i}} = \xi_i ,  \quad i = 1,2,\dots,m-1,
				$$
				and since
				$$
				\frac{T_{i+1,i}}{S_{i+1,i}} = \frac{R_{i+1,i+1} (H_V)_{i+1,i} (R_B)_{i i}}{R_{i+1,i+1} (K_V)_{i+1,i} (R_B)_{i i}} = \frac{(H_V)_{i+1,i}}{(K_V)_{i+1,i}}, \quad  i = 1,2,\dots,m-1,
				$$
				where in the first equality we used a result stated in the proof of Theorem \ref{theorem:struct_BiRat_pair}
				with $R$ and $R_B$ upper-triangular matrices.
			\end{proof}
			
			%\underline{\textbf{Superdiagonals}}\\
			\begin{lemma}\label{lemma:superdiag}
				Let $(T,S)$ be the tridiagonal pencil satisfying \eqref{eq:tridiag_pencil}. The ratio of the superdiagonal elements of $(T,S)$ reveals the (complex conjugate of the) poles $\varPsi = \{\psi_1,\psi_2,\dots,\psi_{m-2} \}$  of $\Ls(A^H,w;\varPsi)$
				\begin{equation}
				\frac{T_{i,i+1}}{S_{i,i+1}} = \bar{\psi}_{i-1}, \qquad i=2,3,\dots, m-1.
				\end{equation}
			\end{lemma}
			\begin{proof}
				Assume we have $W^H A V S = T$ as in Theorem \ref{theorem:struct_BiRat_pair}.
				Note that another tridiagonal pencil $(\widetilde{S},\widetilde{T})$ exists for which
				\begin{equation}\label{eq2}
				V^HA^HW\widetilde{T}=\widetilde{S}.
				\end{equation}
				Equation \eqref{eq:tridiag_pencil} represents projection onto $\K$ and orthogonal to $\Ls$ and Equation \eqref{eq2} represents projection onto $\Ls$ and orthogonal to $\K$. Hence, from Lemma \ref{lemma:subdiag} we know that the ratios of the subdiagonals of $(T,S)$ and $(\widetilde{S},\widetilde{T})$ reveal the poles of $\K$ and $\Ls$, respectively.\\
				Starting from Equation \eqref{eq:tridiag_pencil} and \eqref{eq2} we can relate the matrix pencils as follows
				\begin{equation*}
				\begin{cases}
				W^HAV = TS^{-1}\\
				V^HA^HW = \widetilde{S}\widetilde{T}^{-1}
				\end{cases}
				\Rightarrow
				\begin{cases}
				W^HAV = TS^{-1}\\
				W^HAV = \widetilde{T}^{-H}\widetilde{S}^H
				\end{cases}
				\end{equation*}
				concluding that $TS^{-1} = \widetilde{T}^{-H}\widetilde{S}^H$. Rewriting this equation as $\widetilde{T}^HT = \widetilde{S}^HS$	leads to two pentadiagonal matrices.
				Let us assign a variable to each off-diagonal element, diagonal elements are marked as an $\x$, because these are not relevant for the proof
				\begin{align*}
				\begin{bmatrix}
				\x&\tilde{\tau}_1\\
				\tilde{t}_1 & \x & \tilde{\tau}_2\\
				& \tilde{t}_2 &\x &\ddots \\
				& & \ddots &  \ddots& \tilde{\tau}_{n-1}\\
				& & & \tilde{t}_{n-1} & \x
				\end{bmatrix}^H
				\begin{bmatrix}
				\x&{\tau}_1\\
				{t}_1 & \x & {\tau}_2\\
				& {t}_2 &\x &\ddots \\
				& & \ddots &  \ddots& {\tau}_{n-1}\\
				& & & {t}_{n-1} & \x
				\end{bmatrix}
				= \\
				\begin{bmatrix}
				\x&\tilde{\sigma}_1\\
				\tilde{s}_1 & \x & \tilde{\sigma}_2\\
				& \tilde{s}_2 &\x &\ddots \\
				& & \ddots &  \ddots& \tilde{\sigma}_{n-1}\\
				& & & \tilde{s}_{n-1} & \x
				\end{bmatrix}^H
				\begin{bmatrix}
				\x&{\sigma}_1\\
				{s}_1 & \x & {\sigma}_2\\
				& {s}_2 &\x &\ddots \\
				& & \ddots &  \ddots& {\sigma}_{n-1}\\
				& & & {s}_{n-1} & \x
				\end{bmatrix}.
				\end{align*}
				Hence, by equating the second superdiagonals and second subdiagonals of both pentadiagonal matrices we get
				\begin{align*}
				&\begin{cases}
				t_i\tilde{\tau}_{i+1}^H = s_i \tilde{\sigma}_{i+1}^H\\
				\tau_{i+1}\tilde{t}_{i}^H = \sigma_{i+1} \tilde{s}_{i}^H
				\end{cases}, \qquad i=1,\dots,m-2\\
				\Rightarrow &\begin{cases}
				\xi_i = {t_i}/{s_i} = {\tilde{\sigma}_{i+1}^H}/{\tilde{\tau}_{i+1}^H}\\
				\psi_i = {\tilde{s}_i}/{\tilde{t}_i} = {{\tau}_{i+1}^H}/{{\sigma}_{i+1}^H}\\
				\end{cases}, \qquad i=1,\dots,m-2.
				\end{align*}
				where the last equality uses the result from Lemma \ref{lemma:subdiag}.
			\end{proof}
		
			Note that Lemma \ref{lemma:superdiag} allows for freedom in the choice of $T_{1,2}$ and $S_{1,2}$.
			The results from Theorem \ref{theorem:tridiag_pair}, Lemma \ref{lemma:subdiag} and Lemma \ref{lemma:superdiag} can be used to construct a six-term recurrence relation, which builds biorthogonal bases for rational Krylov subspaces. This is the Lanczos-type iteration, given in Algorithm \ref{alg:RatLan}. The derivation of the iteration is omitted since it is straightforward but lengthy.

\subsection{Numerical experiments}\label{sec:numerics}
	The validity of the Lanczos-type iteration is verified by applying it to solve an eigenvalue problem. 
	Three characteristics of the algorithm will be monitored:
	\begin{itemize}
		\item $\Vert W_n^HV_n - I\Vert_2$, a measure for the biorthogonality of bases $W$ and $V$,
		\item $\Vert W_{n+1}^H A V_{n+1} \underbar{S}_n -\underbar{T}_n \Vert_2$, a measure for the quality of the oblique projection and,
		\item a Ritz plot visualizing the quality of Ritz values as approximations to eigenvalues.
	\end{itemize}
		The projection measure uses the expansion matrices $(\underbar{T}_n,\underbar{S}_n)$ resulting from the Lanczos iteration, i.e., they are of dimension $(n+1)\times n$.
		We compare the Ritz values $\theta^{(n)}$ of $(T_n,S_n)$, last row of $(\underbar{T}_n,\underbar{S}_n)$ removed, with the eigenvalues $\lambda$ of $A$. Ritz plots visualize how close the $n$ Ritz values $\theta^{(n)}_i$, $1\leq i\leq n$, are to the closest eigenvalue $\lambda_i := \underset{\lambda}{\min} \vert\theta^{(n)}_i-\lambda\vert$ for increasing $n$. The colours show how accurate the approximation is:
	\begin{align*}
	\text{red: } &\Vert\theta^{(n)}_i-\lambda_i \Vert_2<10^{-8},\\
	\text{yellow: }&\Vert\theta^{(n)}_i-\lambda_i \Vert_2<10^{-5},\\
	\text{green: }&\Vert\theta^{(n)}_i-\lambda_i \Vert_2<10^{-2},\\
	\text{blue: } &\Vert\theta^{(n)}_i-\lambda_i\Vert_2\geq 10^{-2}.
	\end{align*}

	\begin{example}\label{example:num1}
		Consider a random 50 $\times$ 50 upper triangular matrix with eigenvalues $\lambda_i = i$, $i=1,2,\dots,50$.
		Krylov subspaces $\K(A,v;\Xi)$ and $\Ls(A^H,w;\varPhi)$ are build using $v=w$ and $\Xi = \varPhi = \{0,24.1,0,24.1,\dots \}$.
		Biorthogonal projection using these subspaces lead to Figure \ref{fig:num1_biorth} for the biorthogonality measure, Figure \ref{fig:num1_proj} for the measure quantifying the projection $(T,S)$ and Figure \ref{fig:num1_Ritz} showing the Ritz plot. The Ritz plot clearly shows that convergence is concentrated around the chosen poles $0$ and $24.1$. This is the expected behaviour, the convergence of rational Krylov methods can be focussed on certain parts of the spectrum \cite{Ru84}.

		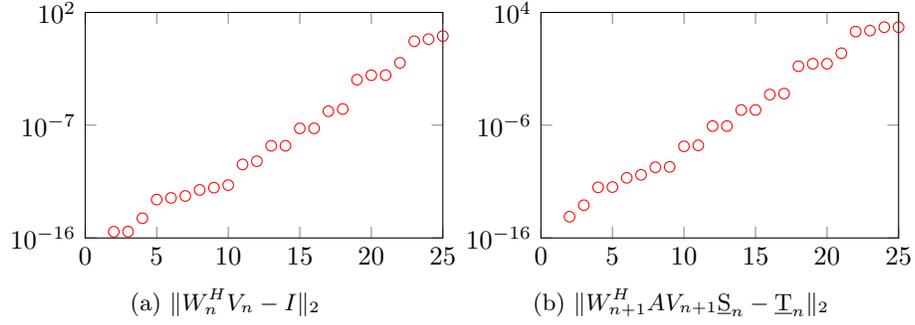
\begin{figure}[!ht]
			\centering
			\begin{subfigure}[b]{0.49\textwidth}
				\setlength\figureheight{3cm}
				\setlength\figurewidth{5cm}
				{% This file was created by matlab2tikz.
\begin{tikzpicture}

\begin{axis}[%
width=0.951\figurewidth,
height=\figureheight,
at={(0\figurewidth,0\figureheight)},
scale only axis,
xmin=0,
xmax=25,
xtick={0,5,10,15,20,25},
ymode=log,
ymin=1e-16,
ymax=1e+02,
yminorticks=true,
axis background/.style={fill=white}
]
\addplot [color=red, draw=none, mark=o, mark options={solid, red}, forget plot]
  table[row sep=crcr]{%
1	0\\
2	3.1463e-16\\
3	3.2398e-16\\
4	3.7502e-15\\
5	1.1612e-13\\
6	1.5703e-13\\
7	2.2401e-13\\
8	6.8067e-13\\
9	1.0358e-12\\
10	1.6388e-12\\
11	7.3848e-11\\
12	1.3336e-10\\
13	2.2668e-09\\
14	2.2828e-09\\
15	5.5837e-08\\
16	5.624e-08\\
17	1.2976e-06\\
18	1.9516e-06\\
19	0.00042841\\
20	0.00098497\\
21	0.00098504\\
22	0.0094687\\
23	0.50598\\
24	0.74475\\
25	1.3053\\
%26	1.3061\\
};
\end{axis}
\end{tikzpicture}%
}
				\caption{$\Vert W_n^HV_n - I\Vert_2$}
				\label{fig:num1_biorth}
			\end{subfigure}
			\begin{subfigure}[b]{0.49\textwidth}
				\setlength\figureheight{3cm}
				\setlength\figurewidth{5cm}
				{% This file was created by matlab2tikz.
\begin{tikzpicture}

\begin{axis}[%
width=0.951\figurewidth,
height=\figureheight,
at={(0\figurewidth,0\figureheight)},
scale only axis,
xmin=0,
xmax=25,
xtick={0,5,10,15,20,25},
ymode=log,
ymin=1e-16,
ymax=1e+04,
yminorticks=true,
axis background/.style={fill=white}
]
\addplot [color=red, draw=none, mark=o, mark options={solid, red}, forget plot]
  table[row sep=crcr]{%
1	5.5511e-17\\
2	7.6417e-15\\
3	8.015e-14\\
4	3.0651e-12\\
5	3.2659e-12\\
6	2.1328e-11\\
7	3.972e-11\\
8	1.8493e-10\\
9	2.013e-10\\
10	1.333e-08\\
11	1.6274e-08\\
12	8.4004e-07\\
13	8.4198e-07\\
14	2.2336e-05\\
15	2.2402e-05\\
16	0.00051915\\
17	0.00063919\\
18	0.1714\\
19	0.28438\\
20	0.28442\\
21	2.4261\\
22	202.42\\
23	246.01\\
24	494.43\\
25	494.57\\
%26	494.6\\
};
\end{axis}
\end{tikzpicture}%
}
				\caption{$\Vert W_{n+1}^H A V_{n+1} \underbar{S}_n -\underbar{T}_n \Vert_2$}
				\label{fig:num1_proj}
			\end{subfigure} 
			\caption{Measures for Example \ref{example:num1}, with $n$ the dimension of the rational Krylov subspaces.}
		\end{figure}
		
		\begin{figure}[!ht]
			\centering
			\includegraphics[clip, trim=6.8cm 20.41cm 6cm 4cm, width=0.5\textwidth]{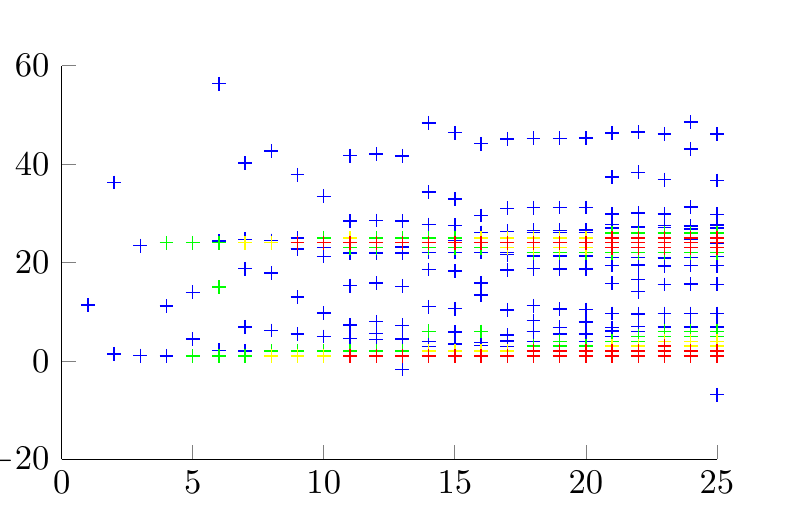}
			\caption{Ritz plot for Example \ref{example:num1}, with $n$ the dimension of the rational Krylov subspaces.}
			\label{fig:num1_Ritz}
		\end{figure}
	\end{example}
	
	\begin{example}\label{example:num2}
		To show the connection between convergence of eigenvalues and biorthogonality, we choose a pole closer to an eigenvalue. This leads to faster convergence to this eigenvalue and thus to faster loss of biorthogonality \cite{Pa71}. The poles chosen now are $\Xi = \varPhi = \{0,24+10^{-5},0,24+10^{-5},\dots \}$. Figure \ref{fig:num2_Ritz} shows that eigenvalue $\lambda=24$ is found in fewer iterations than in Example \ref{example:num1}. This leads to faster loss of biorthogonality of the bases, shown on Figure \ref{fig:num2_biorth}. Figure \ref{fig:num2_proj} shows that the quality of the projection is related to the biorthogonality of the bases.

		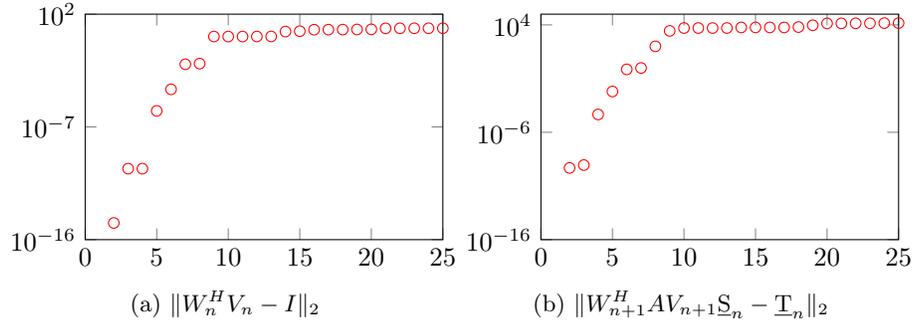
\begin{figure}[!ht]
			\centering
			\begin{subfigure}[b]{0.49\textwidth}
				\setlength\figureheight{3cm}
				\setlength\figurewidth{5cm}
				{% This file was created by matlab2tikz.
\begin{tikzpicture}

\begin{axis}[%
width=0.951\figurewidth,
height=\figureheight,
at={(0\figurewidth,0\figureheight)},
scale only axis,
xmin=0,
xmax=25,
xtick={0,5,10,15,20,25},
ymode=log,
ymin=1e-16,
ymax=1e+02,
yminorticks=true,
axis background/.style={fill=white}
]
\addplot [color=red, draw=none, mark=o, mark options={solid, red}, forget plot]
  table[row sep=crcr]{%
1	0\\
2	2.1566e-15\\
3	4.6487e-11\\
4	4.6488e-11\\
5	1.909e-06\\
6	0.0001\\
7	0.010024\\
8	0.011337\\
9	1.6418\\
10	1.6419\\
11	1.6419\\
12	1.6484\\
13	1.6579\\
14	4.1088\\
15	4.3799\\
16	5.7549\\
17	5.8309\\
18	5.8556\\
19	6.0092\\
20	6.1575\\
21	7.3856\\
22	7.4227\\
23	7.4625\\
24	7.4996\\
25	7.5663\\
%26	10.022\\
};
\end{axis}
\end{tikzpicture}%
}
				\caption{$\Vert W_n^HV_n - I\Vert_2$}
				\label{fig:num2_biorth}
			\end{subfigure}
			\begin{subfigure}[b]{0.49\textwidth}
				\setlength\figureheight{3cm}
				\setlength\figurewidth{5cm}
				{% This file was created by matlab2tikz.
\begin{tikzpicture}

\begin{axis}[%
width=0.951\figurewidth,
height=\figureheight,
at={(0\figurewidth,0\figureheight)},
scale only axis,
xmin=0,
xmax=25,
xtick={0,5,10,15,20,25},
ymode=log,
ymin=1e-16,
ymax=1e+05,
yminorticks=true,
axis background/.style={fill=white}
]
\addplot [color=red, draw=none, mark=o, mark options={solid, red}, forget plot]
  table[row sep=crcr]{%
1	1.7347e-17\\
2	4.7801e-10\\
3	8.8506e-10\\
4	4.582e-05\\
5	0.0063307\\
6	0.7239\\
7	0.97311\\
8	101.14\\
9	2765.6\\
10	5114.6\\
11	5114.6\\
12	5122.2\\
13	5132.4\\
14	5820.5\\
15	5820.6\\
16	5834.3\\
17	5834.3\\
18	6310.7\\
19	9263.4\\
20	13859\\
21	13859\\
22	14108\\
23	14247\\
24	14741\\
25	14741\\
%26	14745\\
};
\end{axis}
\end{tikzpicture}%
}
				\caption{$\Vert W_{n+1}^H A V_{n+1} \underbar{S}_n -\underbar{T}_n \Vert_2$}
				\label{fig:num2_proj}
			\end{subfigure} 
			\caption{Measures for Example \ref{example:num2}, with $n$ the dimension of the rational Krylov subspaces.}
		\end{figure}
		%	
		%		\begin{figure}[!ht]
		%			\centering
		%			\setlength\figureheight{3.5cm}
		%			\setlength\figurewidth{6cm}
		%			{\input{figs/num/biorth2.tikz}}
		%		\caption{Biorthogonality measure $\Vert W_n^HV_n - I\Vert_2$ for Example \ref{example:num2}, with $n$ the dimension of the rational Krylov subspaces.}
		%			\label{fig:num2_biorth}
		%		\end{figure}
		%		
		%		\begin{figure}[!ht]
		%			\centering
		%			\setlength\figureheight{3.5cm}
		%			\setlength\figurewidth{6cm}
		%			{\input{figs/num/proj2.tikz}}
		%			\caption{Projection measure $\Vert W_{n+1}^H A V_{n+1} \underbar{S}_n -\underbar{T}_n \Vert_2$ for Example \ref{example:num2}, with $n$ the dimension of the rational Krylov subspaces.}
		%			\label{fig:num2_proj}
		%		\end{figure}
		%		
		\begin{figure}[!ht]
			\centering
			\includegraphics[clip, trim=6.8cm 20.41cm 6cm 4cm, width=0.5\textwidth]{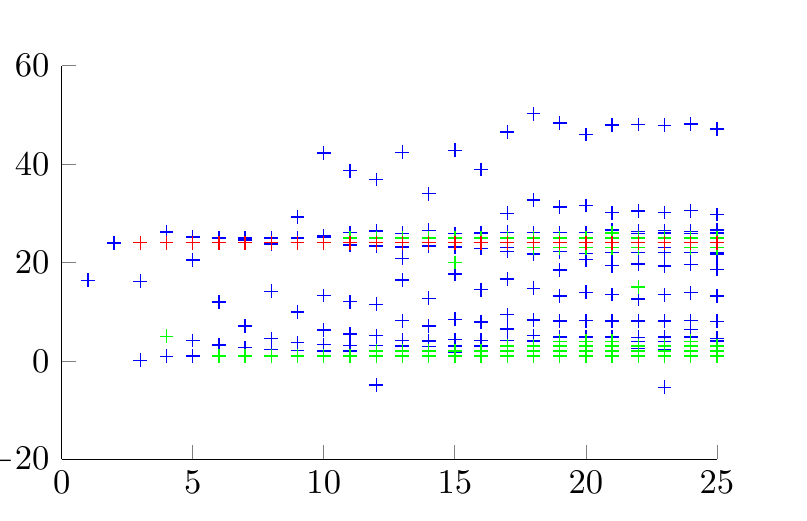}
			\caption{Ritz plot for Example \ref{example:num2}, with $n$ the dimension of the rational Krylov subspaces.}
			\label{fig:num2_Ritz}
		\end{figure}
		
	\end{example}

	Note that we did not use examples where $\Xi\neq\varPhi$, since the behaviour of such choices is not comparable with any existing Lanczos-type iterations and subject to future research.
	From Example \ref{example:num1} and Example \ref{example:num2} we conclude that the novel Lanczos-type iteration exhibits the expected behaviour, i.e., comparable to the behaviour of known iterations. Hence, the validity of the rational Lanczos iteration is substantiated.
	
	\section{Conclusion}
	A general framework to predict the various structures arising in the context of rational Krylov subspace methods is developed.
	From this framework a pair of short recurrence relations for biorthogonal bases of rational Krylov subspaces is deduced. Based on the short recurrence relation a Lanczos-type iteration is derived which constructs these bases together with a tridiagonal pencil representing the obliquely projected matrix.
	The framework generalizes many classical and more recent results, as does the novel rational Lanczos iteration.
	Numerical tests are performed as a proof of concept for the iteration.
	\section{Future Research}
	A proper analysis of the numerical behaviour of the Lanczos iteration could provide the means to make it more stable.\\
	Furthermore the relation between biorthogonal rational Krylov subspaces and biorthogonal functions will be looked into.

	\appendix
		\clearpage	
		\fakesection{}
	\thispagestyle{empty}
		\begin{algorithm}
			\begin{algorithmic}[1]
				\Procedure{RatLan}{$A,v,w,n,\{b_i\}_{i=2}^{n},\{l_i\}_{i=2}^{n},\{\lambda_i\}_{i=2}^{n},\{\beta_i\}_{i=2}^{n}$}%\Comment{The g.c.d. of a and b}
				% Initialize
				\State $\lambda_1$ and $\beta_1$ can be chosen arbitrary, but not simultaneously 0
				\State $l_1$ and $b_1$ can be chosen arbitrary, but not simultaneously 0
				\State $\bv_1\gets (l_2 A-b_2 I)^{-1}v$
				\State $\bw_1\gets (\beta_2 A^H-\lambda_2 I)^{-1}w$
				\State $\temp_1\gets (\bv_1,w)/(A\bv_1,w)$
				\State $\temp_2\gets (\bw_1,v)/(\bw_1,Av)$
				\State $\hv_2 \gets -\temp_1 A \bv_1 + \bv_1$
				\State $\hw_2 \gets -\temp_2 A^H \bw_1 + \bw_1$
				
				\State $\bar{\sigma} \bar{\gamma_{1}} r d_1 = 1/(\hv_2,\hw_2)$ 
				\State $l_2 \gets l_2/r$
				\State $b_2 \gets b_2/r$
				\State $\beta_2 \gets \beta_2/\sigma$
				\State $\lambda_2 \gets \lambda_2/\sigma$
				\State $c_1\gets d_1\temp_1$
				\State $\delta_1\gets \gamma_1\temp_2$
				\State $v_2 \gets r d_1 \hv_2$
				\State $w_2 \gets \sigma \gamma_1 \hw_2$
				% Iterate
				\For{$i=2:n-1$}
				\If {$\beta_{i-1}\neq 0$}
				\State $\tv_{i-1} = 1/{\bar{\beta}_{i-1}} (l_{i+1} A-b_{i+1}I)^{-1} (\bar{\beta}_{i-1} A -\bar{\lambda}_{i-1}I) v_{i-1} $
				\Else
				\State $\tv_{i-1} = - (l_{i+1} A-b_{i+1}I)^{-1} \lambda_{i-1} v_{i-1}$ 
				\EndIf
				
				\If {$l_{i-1}\neq 0$}
				\State $\tw_{i-1} = 1/{\bar{l}_{i-1}} (\beta_{i+1} A^H-\lambda_{i+1}I)^{-1} (\bar{l}_{i-1} A^H -\bar{b}_{i-1}I) w_{i-1} $ 
				\Else
				\State $\tw_{i-1} = -(\beta_{i+1} A^H-\lambda_{i+1}I)^{-1} \bar{b}_{i-1} w_{i-1} $ 
				\EndIf

				\State $\bv_i \gets (l_{i+1}A - b_{i+1} I)^{-1} v_i$
				\State $\bw_i \gets (\beta_{i+1}A^H - \lambda_{i+1} I)^{-1} w_i$
				
				\State $\tempv_1 \gets \frac{(\tv_{i-1},w_{i-1}) (A\bv_{i} ,w_i) - (\tv_{i-1},w_i)(A\bv_i,w_{i-1})}{(\bv_{i},w_{i-1}) (A\bv_{i} ,w_i) - (\bv_{i},w_i)(A\bv_i,w_{i-1})} $
				\State $\tempv_2 \gets \tempv_1 \frac{(\bv_i,w_i)}{(A\bv_i,w_i)} - \frac{(\tv_{i-1},w_i)}{(A\bv_i,w_i)}$
				\State $\tempw_1 \gets \frac{(\tw_{i-1},v_{i-1}) (\bw_{i} ,Av_i) - (\tw_{i-1},v_i)(\bw_i,Av_{i-1})}{(\bw_{i},v_{i-1}) (\bw_{i} ,Av_i) - (\bw_{i},v_i)(\bw_i,Av_{i-1})} $
				\State $\tempw_2 \gets \tempw_1 \frac{(\bw_i,v_i)}{(\bw_i,Av_i)} - \frac{(\tw_{i-1},v_i)}{(\bw_i,Av_i)}$
				\State $\hv_{i+1} \gets -\tempv_2 A \bar{v}_i + \tempv_1 \bar{v}_i - \tv_{i-1}$
				\State $\hw_{i+1} \gets -\tempw_2 A^H \bar{w}_i + \tempw_1 \bar{w}_i - \tw_{i-1}$
				\State $\bar{\alpha}_i u_i = 1/(\hv_{i+1},\hw_{i+1})$ 
				
				\State $v_{i+1}\gets u_i \hv_{i+1}$
				\State $w_{i+1}\gets \alpha_i \hw_{i+1}$
				\State $d_i \gets \tempv_1 u_i$
				\State $c_i \gets \tempv_2 u_i$
				\State $\gamma_i \gets \tempw_1 \alpha_i$
				\State $\delta_i \gets \tempw_2 \alpha_i$
				\EndFor
				\State $V_n \gets [v,v_2,\dots, v_n]$
				\State $W_n \gets [w,w_2,\dots, w_n]$  \Comment{Biorthogonal bases $W_n^HV_n=I$}
				\EndProcedure
			\end{algorithmic}
			\caption{Nonhermitian rational Lanczos iteration}
			\label{alg:RatLan}
	\end{algorithm}

	\clearpage
	\bibliographystyle{siam}
	\bibliography{references}
	
\end{document}